\documentclass[12pt]{amsart}
\pdfoutput=1
\usepackage{amsmath, amssymb, amsthm}
\usepackage{rsfso}
\usepackage[dvipsnames,svgnames]{xcolor}
\usepackage{tikz}
\usetikzlibrary{arrows,cd}
\tikzset{>=latex}
\usepackage{mathtools}
\usepackage{paralist}
\usepackage[alphabetic]{amsrefs}
\usepackage[T1]{fontenc}
\usepackage{mathptmx}
\usepackage{microtype}
\usepackage{tikz}
\usepackage[colorlinks=true,linkcolor=blue,urlcolor=blue, citecolor=blue]{hyperref}
\usepackage[paperwidth=8.5in, paperheight=11in, inner=2.5cm, outer=2.5cm, top=2.5cm, bottom=2.5cm]{geometry}
\newtheorem{lemma}{Lemma}[section]
\newtheorem{theorem}[lemma]{Theorem}
\newtheorem{corollary}[lemma]{Corollary}
\newtheorem{proposition}[lemma]{Proposition}
\newtheorem*{theoremA}{Theorem~\ref{t:ineq}}
\newtheorem*{theoremB}{Theorem~\ref{t:>tor}}
\newtheorem*{theoremC}{Theorem~\ref{t:delPezzo}}
\newtheorem*{theoremD}{Theorem~\ref{t:ruled}}
\theoremstyle{definition}
\newtheorem{definition}[lemma]{Definition}
\newtheorem{remark}[lemma]{Remark}
\newtheorem{example}[lemma]{Example}
\AtBeginEnvironment{example}{\pushQED{\qed}}
\AtEndEnvironment{example}{\popQED\endexample}

\newcommand{\relphantom}[1]{\mathrel{\phantom{#1}}}

\newcommand{\ideal}[1]{\ensuremath{\left\langle #1 \right\rangle}}
\newcommand{\coloneq}{\ensuremath{\mathrel{\mathop :}=}}
\newcommand{\CC}{\ensuremath{\mathbb{C}}} 
\newcommand{\NN}{\ensuremath{\mathbb{N}}}
\newcommand{\PP}{\ensuremath{\mathbb{P}}} 
 
\newcommand{\RR}{\ensuremath{\mathbb{R}}} 
\newcommand{\ZZ}{\ensuremath{\mathbb{Z}}} 
\newcommand{\sO}{\ensuremath{\mathcal{O}\kern-1.75pt}}
\newcommand{\ii}{\ensuremath{\mathrm{i}}}
\DeclareMathOperator{\Bl}{Bl}
\DeclareMathOperator{\codim}{codim}

\DeclareMathOperator{\conv}{conv}
\DeclareMathOperator{\ev}{ev\kern-0.5pt}
\DeclareMathOperator{\Ker}{Ker}
\DeclareMathOperator{\Pic}{Pic}
\DeclareMathOperator{\Proj}{Proj}
\DeclareMathOperator{\Spec}{Spec}
\DeclareMathOperator{\variety}{V}
\DeclareMathOperator{\mcrk}{mclen}
\DeclareMathOperator{\rrank}{clen}

\begin{document}

\vspace*{-1.0em}

\title[Nonnegativity on Surfaces]%
  {Nonnegativity Certificates on Real Algebraic Surfaces}

\author[G.~Blekherman]{Grigoriy Blekherman}
\address{Grigoriy Blekherman: School of Mathematics, 
    Georgia Tech, 686 Cherry Street, Atlanta, 
    Georgia, 30332, United States of America; %
  {\normalfont \texttt{greg@math.gatech.edu}}}

\author[R.~Sinn]{Rainer Sinn}
\address{Rainer Sinn: Mathematisches Institut, 
    Universit\"{a}t Leipzig, PF 1009 20, 
    04009 Leipzig, Germany; %
    {\normalfont \texttt{rainer.sinn@uni-leipzig.de}}}

\author[G.G.~Smith]{Gregory G.{} Smith}
\address{Gregory G.{} Smith: Department of Mathematics and Statistics, 
    Queen's University, Kingston, 
    Ontario, K7L 3N6, Canada; %
    {\normalfont \texttt{ggsmith@mast.queensu.ca}}}

\author[M.~Velasco]{Mauricio Velasco} 
\address{Mauricio Velasco: Departamento de Inform{\'a}tica, 
    Universidad Cat{\'o}lica del Uruguay, 
    Av.{} 8 de Octubre 2738, 
    11600 Montevideo, Uruguay; %
    {\normalfont \texttt{mauricio.velasco@ucu.edu.uy}}}

% \keywords{convex algebraic geometry, sums of squares}

%\date{2024-09-13}

\begin{abstract}
  We introduce tools for transferring nonnegativity certificates for global
  sections between line bundles on real algebraic surfaces. As applications,
  we improve Hilbert's degree bounds on sum-of-squares multipliers for
  nonnegative ternary forms, give a complete characterization of nonnegative
  real forms of del Pezzo surfaces, and establish quadratic upper bounds for
  the degrees of sum-of-squares multipliers for nonnegative forms on real
  ruled surfaces.
\end{abstract}

\maketitle

%%%%%%%%%%%%%%%%%%%%%%%%%%%%%%%%%%%%%%%%%%%%%%%%%%%%%%%%%%%%%%%%%%%%%%%%%%%%%%
\vspace*{-1.0em}
\section{Overview}

\noindent
Characterizing nonnegativity is a fundamental problem in both real algebraic
geometry and optimization. We develop a \emph{transfer} approach to testing
and certifying nonnegativity of polynomials. The key to this approach is the
following simple observation: If $f,g$ are multivariate polynomials satisfying
the equation $fg=s$ where $s$ is a sum-of-squares of polynomials then
nonnegativity of $f$ is equivalent to nonnegativity of $g$. The identity
$fg=s$ thus \emph{transfers} the problem of testing the nonnegativity of $f$
to that of $g$. We are interested in transferring nonnegativity certificates
between \emph{classes of functions}: if we can show that every nonnegative
function $f$ in a certain class has a multiplier $g$ in a different class,
such that $fg$ is a sum of squares, then we have transferred nonnegativity
testing from the class of $f$ to the class of $g$. If the class of $g$ is
simpler in an appropriate sense, then we can iteratively apply the transfer
procedure aiming to reduce the problem to a class of functions where
nonnegativity is well understood.

We carry out the program outlined above for forms on \emph{real projective
  surfaces}. Our method transfers nonnegativity certificates for sections of a
certain line bundle on a surface, to nonnegativity certificates of sections of
a ``simpler'' line bundle, provided the bundles satisfy certain cohomological
inequalities. These inequalities hold for line bundles on a wide array of
algebraic surfaces, including rational and more generally ruled surfaces, and
allow us to fully characterize nonnegativity sections in several situations.

One of the least understood aspects of the relationship between nonnegative
polynomials and sums of squares is the question of \emph{degree bounds} when
writing nonnegative polynomials as \emph{sums of squares of rational
  functions}. Hilbert's 17th problem asked whether every globally nonnegative
polynomial $f$ can be written as a sum of squares of rational functions. It is
easy to see that being able to write a nonnegative polynomial $f$ in this way
is equivalent to the existence of a sum-of-squares multiplier $g$ and a
sum-of-squares $s$ such that $fg=s$ as above. Artin's affirmative solution to
Hilbert's 17th problem in \cite{Artin} thus transfers nonnegativity
certification into a sum-of-squares feasibility problem. However, our
understanding of bounds on the degree of the available multipliers $g$ (both
upper and lower) remains quite poor.

Prior to posing the 17th problem, Hilbert showed in 1893 \cite{Hilbert93} that
the result holds for ternary forms (homogeneous polynomials in three
variables). Hilbert's original proof of the case of ternary forms came with
upper bounds on the degree of the sum-of-squares multipliers, and these bounds
remained unimproved until our current work. His proof iteratively lowers the
degree of the ternary form for which nonnegativity has to be certified.
Hilbert's iterative approach is an inspiration for the approach introduced in
this article: we extend Hilbert's methods from the case of $\PP^2$ to general
algebraic surfaces, developing a general transfer theory. We then specialize
our results to several types of surfaces. First, we focus on toric surfaces
and provide a combinatorial interpretation for our transfer Theorem in terms
of lattice points in polygons. The freedom to choose different line bundles
allowed by the transfer Theorem leads us to revisit Hilbert's bounds for
ternary forms, obtaining the first improvement since the original paper in
1893. Furthermore, in the first case where degree bounds for multipliers of
ternary forms are not known, namely ternary forms of degree 10, we prove a
tight upper bound; this is the first new instance since Hilbert's work of an
exact multiplier degree bound for his 17th problem.

We then focus on real del Pezzo surfaces where we take advantage of the
classification of their real structures~\cite{Rus}. The resulting
understanding of the real Picard group allows us to fully carry out the
transfer program outlined above, obtaining a complete classification of
nonnegative sections for all line bundles. The del Pezzo case is particularly
interesting since it illustrates the role that real structures play in
determining the available multipliers and the final form of the nonnegativity
certificates. Furthermore, this is the first example of concrete geometric
bounds for nonnegativity certificates that apply to a family of surfaces
having a non-trivial moduli space.

Finally we develop an asymptotic theory of degree bounds for surfaces. The
basic question is the following: given a fixed real algebraic surface $X$, can
we bound $k$ such that every form $f$ of degree $2d$ has a multiplier $g$ of
degree $2k$ so that $fg$ is a sum of squares? We show that for most
nonsingular ruled surfaces and $d$ large enough, the degree of multipliers $k$
is bounded from above by a quadratic function in $d$. This is the first result
on multiplier degree bounds that applies to non-rational surfaces.

%%----------------------------------------------------------------------------
\subsection*{Main results}

Let $X$ be a totally-real variety. We say that a divisor $E$ \emph{supports
  multipliers} for a divisor $D$ if, for any nonnegative global section $f$ in
$\smash{H^0 \kern-1.0pt\bigl( X, \sO_{X}(2D) \kern-1.0pt \bigr) \kern-1.0pt}$,
there exists a nonzero global section $g$ in
$\smash{H^0 \kern-1.0pt\bigl( X, \sO_{X}(2E) \kern-1.0pt \bigr) \kern-1.0pt}$
such that the product $fg$ in
$\smash{H^0 \kern-1.0pt\bigl( X, \sO_{X}(2D + 2E) \kern-1.0pt \bigr)
  \kern-1.0pt}$ is a sum of squares. Equivalently, we can transfer testing
nonnegativity of global sections in
$\smash{H^0 \kern-1.0pt\bigl( X, \sO_{X}(2D) \kern-1.0pt \bigr) \kern-1.0pt}$,
to testing nonnegativity of global sections in
$\smash{H^0 \kern-1.0pt\bigl( X, \sO_{X}(2E) \kern-1.0pt \bigr) \kern-1.0pt}$.

Our main technical contribution is the following theorem.

\begin{theoremA}
  Assume that $X$ is a totally-real geometrically-integral projective surface.
  Let $D$ and $E$ be divisors on $X$ with $D$ free (equivalently, the line
  bundle $\sO_{X}(D)$ is globally generated), $D+E$ very ample, and
  $\smash{H^0 \kern-1.0pt \bigl( X, \sO_{X}(E-D) \kern-1.0pt \bigr)
    \kern-1.5pt}
  = \smash{H^1 \kern-1.0pt \bigl( X, \sO_{X}(D+E) \kern-1.0pt
    \bigr)\kern-1.5pt} = \smash{H^1 \kern-1.0pt \bigl( X, \sO_{X}(2E)
    \kern-1.0pt \bigr) \kern-1.5pt} = 0$. The inequality
    \[
        h^0(X, D+E) 
        > 1 + \left\lceil
        \frac{h^0(X, 2D+2E) - h^0(X, 2E) - h^0(X, D+E) - h^1(X, E-D)}%
        {2} \right\rceil
    \]
    implies that the divisor $E$ supports multipliers for the divisor $D$.
\end{theoremA}

The main inequalities and cohomological vanishing conditions can sometimes be
simplified. We first consider toric surfaces, where we develop a criterion for
transferring nonnegativity of Laurent polynomials with support in a lattice
polygon $2P$ to Laurent polynomials with support in a lattice polygon $2Q$. We
need the following terminology: If $A\subseteq \RR^2$ then the number of
\emph{reduced connected components of $A$} is one less than the number of
connected components of $A$ and a \emph{lattice translate of $A$} is a set of
the form $A+m$ for $m \in \ZZ^2$. We write $\#A$ for the number of lattice
points contained in $A$ and $A^{\circ}$ for the interior of $A$.

\begin{theoremB}
  Assume that $P$ and $Q$ are convex lattice polygons such that no lattice
  translate of $P$ is contained in $Q$. Let $h$ be the total number of reduced
  connected components of the set differences $P \setminus Q'$ as $Q'$ ranges
  over all lattice translates of $Q$. The inequality
    \[
        \#(2Q)+ h > \# \bigl( \kern-1.0pt (P+Q)^{\circ} \kern-1.0pt \bigr)
    \]
    implies that $Q$ supports multipliers for $P$ (i.e. for every nonnegative
    Laurent polynomial $f$ with monomial support in $2P$ there exists a
    Laurent polynomial $g$ with monomial support in $2Q$ such that $fg$ is a
    sum of squares of Laurent polynomials with monomial support in $2(P+Q)$).
  \end{theoremB}

\noindent
This allows us prove sharp degree bounds for degree 10 ternary forms in
Example~\ref{e:hb1}, and improve Hilbert's bounds for ternary forms in
Example~\ref{e:hb2}.

The inequality in Theorem~\ref{t:ineq} can be rewritten geometrically, in
terms of the adjoint bundle as \[
    h^0(X, 2E) + h^1(X, E-D) > h^0(X, K_X+D+E) \, .
\]
A negative canonical bundle makes the right hand side of this inequality
smaller, making it natural to focus on surfaces with large anticanonical
divisor. Therefore, we look at del Pezzo surfaces in detail. Our main result
in this direction is the following:

\begin{theoremC}
  Let $X$ be a totally-real del Pezzo surface having degree at least $3$ and
  canonical divisor $K_X$. For any nonzero real effective divisor $D$ on $X$,
  there exists a finite sequence $D_0, D_1, \dotsc, D_k$ of effective divisors
  on $X$ with $D_0 = D$ such that $-K_{X} \cdot D_{i} < -K_{X} \cdot D_{i-1}$,
  $D_{i}$ supports multipliers for $D_{i-1}$ for any
  $1 \leqslant i \leqslant k$, and $D_k$ is either zero or a positive multiple
  of a conic bundle. In particular, the length $k$ of the sequence is bounded
  above by $-K_{X} \cdot D$.
\end{theoremC}

\noindent
This theorem allows us to find certificates of nonnegativity on del Pezzo
surfaces as explained in Remark~\ref{r:modSOS}.

Next we also establish asymptotic degree bounds for some embedded surfaces.
Let $X$ in $\PP^n$ be a totally-real surface with canonical divisor $K_{X}$.
Let $A$ be the divisor defined by the hyperplane section of the embedding. Our
goal is to prove degree bounds for certifying nonnegativity of global sections
of $H^0(X,2dA)$ for large $d$. Our main result is that nonnegativity transfer
is possible via the surface $Z$ obtained by blowing-up $X$ at a real point
when $-K_{X} \cdot A > 0$, which implies that the surface $X$ must be ruled.
More precisely, we prove the following:

\begin{theoremD} 
  Assume that $X$ is a totally-real smooth surface with a very ample divisor
  $A$ satisfying $-K_{X} \cdot A > 0$. Let
  $\pi \colon Z \coloneq \Bl_{p}(X) \to X$ be the blow-up of $X$ at a real
  point $p$ and set $H \coloneq \pi^*(A)$. Fix $s$ to be the smallest positive
  integer such that $s (-K_{X} \cdot A) > A \cdot (A + K_{X})$ and choose a
  positive integer $t$ such that
  $\tfrac{1}{2} + \tfrac{1}{3} + \dotsb + \tfrac{1}{t+1} > 2(1 + \sqrt{s})$.
  For all sufficiently large integers $d$, there exists an $(t+1)$-step
  transfer on $Z$ from $d H$ to $(d-1) H$. \end{theoremD}

It follows that the degree of sum of square multipliers can be bounded by a
quadratic polynomial in $d$, i.e. there exists a sum of squares
$g \in H^0(X,2kA)$, such that $fg$ is a sum of squares, and $k$ can be bounded
by a quadratic function in $d$ (see Corollary~\ref{c:quad}).

%%----------------------------------------------------------------------------
\subsection*{Relationship with earlier work} 

The statement that a globally nonnegative polynomial $f$ is a sum of squares
of rational functions is equivalent (by clearing denominators) to saying that
there exists a real polynomial $r$ such that $fr^2$ is a sum of squares. As
mentioned earlier this is in fact equivalent to saying that there exists a sum
of squares $g$, such that $fg$ is a sum of squares. When referring to degree
bounds, we always refer to the bounds on the degree of the sum of squares
multiplier $g$. It is worth noting that Artin's original proof did not produce
any degree bounds. Currently, the best known upper bounds due to Lombardi,
Perrucci and Roy are a multiple tower of exponentials \cite{LPR}, while the
best known lower bounds are linear \cite{BGP}.

The results of \cite{LPR} imply multiple tower of exponentials degree bounds
for real varieties and, more generally, for semialgebraic sets. In \cite{BSV}
tight degree bounds were proved for real curves, which significantly improve
on the general bounds from real algebraic geometry. However, the method does
not transfer to varieties of higher dimension, making surfaces a very
interesting natural next step for a quantitative understanding of
nonnegativity on varieties.

Finally, we would like to mention a related line of work which uses powers of
a fixed polynomial as a multiplier. A result of Reznick \cite{RezUni} shows
that if $f$ is a strictly positive homogeneous polynomial (form), then for
some large enough $r$ we have that $(x_1^2+\dots+x_n^2)^r f$ is a sum of
squares. The paper includes an upper bound on $r$ in terms of the degree, the
number of variables of $f$, and the minimum value of $f$ on the unit sphere.
It is crucial to note that dependence of $r$ on the minimum of $f$ cannot be
removed, and so-called ``uniform denominators" cannot work for all forms of
degree $2d$, as demonstrated in \cite{RezNonUni}. Reznick's result was later
generalized by Scheiderer \cite{ScheidPos}, who showed that if $f$ and $g$ are
both strictly positive forms, then for $k$ large enough we have $fg^k$ is a
sum of squares. However, in this generality, there is no estimate on the size
of $k$. There is also a large literature on Positivstellensatz theorems on
compact affine varieties which do not use multipliers \cite{MarshallB}.
However, it is a feature of these theorems that uniform degree bounds are
simply not possible in general \cite{ScheidReg}.

%%----------------------------------------------------------------------------
\subsection*{Structure of the paper}

Let $Y \subset \PP^n$ be a real, projective, linearly normal curve with graded
coordinate ring $R$. One of our main technical tools is establishing bounds so
that a linear functional $\ell$ on $R_2$ can be written as a sum of few
evaluations on points of $Y$, and the points of $Y$ are chosen in a
conjugate-invariant way. We call the least number of conjugate-invariant point
evaluations the \emph{conjugate invariant length} of $\ell$. The main result
of Sections 2 and 3 is a bound on the (maximal typical) conjugate invariant
length. If the curve $Y$ has no real points, then there is a trivial bound of
$\lceil \frac{1}{2} \dim R_2 \rceil$ of complex pairs of evaluations. We
modestly improve it to $\lceil \frac{1}{2} (\dim R_2-\dim R_1) \rceil+1$, but
this improvement is crucial. The basic idea is simple: given a generic linear
functional $\ell \in R_2^*$, consider the associated quadratic form
$\varphi_{\ell}$. We can make $\varphi_{\ell}$ drop rank by 1 by adding a
multiple of some complex point evaluation. The resulting linear functional
$\ell'$ can be written in terms of point evaluations of
$\Ker \varphi_{\ell'} \cap Y$ (this uses linear normality of $Y$). Then we
just apply the trivial bound to $\Ker Q_{\ell'} \cap Y$, so the rank of
$\ell'$ is at most $\lceil \frac{1}{2} (\dim R_2-\dim R_1) \rceil$, and so the
length of $\ell$ is at most $\lceil \frac{1}{2} (\dim R_2-\dim R_1) \rceil+1$.

Cohomological conditions on $X$ allow us to pass to (and count dimensions in)
the normalization of $Y = V(f)\cap X$. Interestingly, passing to the
normalization $Y'$ of $Y$ increases the dimension of $R_1$, while keeping the
dimension of $R_2$ the same, which improves the effectiveness of our bound on
conjugate symmetric length
$\lceil \frac{1}{2} (\dim R_2(Y') - \dim R_1(Y')) \rceil+1$

In Section~\ref{s:mainproof}, we prove Theorem \ref{t:ineq} by applying results
on conjugate-invariant length, which allow us to count signs in a suitably
defined real quadratic form. We derive a contradiction to non-existence of sum
of squares multipliers in some degree, since the quadratic form is supposed to
be positive definite, and yet the results on the conjugate-invariant length
allow us to show that it doesn't have enough positive eigenvalues.

In the second half of the paper we focus on the applications of
Theorem~\ref{t:ineq}. In Section~\ref{s:toric}, we apply it to toric surfaces
and derive the sharpest known degree bounds for sum of squares multiplies for
ternary forms, improving Hilbert's 1893 result. In Section~\ref{s:pezzos} we
work with real del Pezzo surfaces, and in Section~\ref{s:asymp} we focus on
asymptotic degree bounds for ruled surfaces and prove the existence of
asymptotic quadratic multiplier bounds for them.

%%%%%%%%%%%%%%%%%%%%%%%%%%%%%%%%%%%%%%%%%%%%%%%%%%%%%%%%%%%%%%%%%%%%%%%%%%%%%%
\section{Point Evaluations}

\noindent
This section introduces a numerical invariant of a real projective subvariety
called the maximum typical conjugation-invariant length; see
Definition~\ref{d:mcr}. It is derived from the conjugation-invariant length of
linear functionals, which depends on expressing linear functionals as
conjugation-invariant linear combinations of point evaluations. This subtle
invariant allows one to bound the number of positive and negative of the
eigenvalues for quadratic forms on the subvariety.

Fix a nonnegative integer $n$ and consider an $(n+1)$-dimensional real vector
space $V$. Let $x_0, x_1, \dotsc, x_n$ be a basis for its real dual space
$V^*$ and set
$S \coloneq \operatorname{Sym}(V^*) \cong \RR[x_0, x_1, \dotsc, x_n]$ where
$\deg(x_j) = 1$ for all $0 \leqslant j \leqslant n$. A real projective
subvariety $X$ of $\PP^n \coloneq \Proj(S)$ is an integral closed subscheme
over $\RR$ such that the structure morphism $X \to \Spec(\RR)$ is separated
and of finite type. The saturated homogeneous ideal of the subscheme $X$ in
the polynomial ring $S$ is denoted by $I$ and the homogeneous coordinate ring
of $X$ is the quotient $R \coloneq S/I$. For all integers $j$, the graded
component $R_j$, consisting of all homogeneous elements in $R$ having degree
$j$, is a finite-dimensional real vector space.

We begin by describing a correspondence between the linear functionals
$\ell \colon R_2 \to \RR$ and certain quadratic forms. Since $R_2 = S_2/I_2$,
the pullback of the canonical surjection $\eta \colon S_2 \to R_2$, which
sends a linear functional $\ell \colon R_2 \to \RR$ to the composite map
$\ell \circ \eta \colon S_2 \to \RR$, is injective and defines an isomorphism
between $R_2^*$ and
$I_2^{\perp} \coloneq \{ \psi \in S_2^* \mathrel{|} \text{$\psi(g) = 0$ for
  all $g \in I_2$} \}$. The corresponding quadratic form
$\varphi_{\ell} \colon S_1 \to \RR$ is defined, for all $f$ in $S_1$, by
$\varphi_{\ell}(f) = (\ell \circ \eta)(f^2)$. The kernel of the quadratic form
$\varphi_{\ell}$ is the kernel of its associated symmetric matrix: it is the
linear subspace consisting of all polynomials $f$ in $S_1$ such that
$(\ell \circ \eta)(f \, S_1) = 0$. The \emph{corank} of $\varphi_\ell$ is the
dimension of its kernel.

\begin{example}
  Let $C$ be the complete intersection curve in $\PP^3$ whose homogeneous
  coordinate ring is
  \[
    R \coloneq \RR[x_0 ,x_1, x_2, x_3] /
    \langle x_0^2 + x_1^2, x_1^{} x_2^{} - x_2^2 - x_3^2 \rangle \, .
  \]
  Any linear functional $\ell \colon R_2 \to \RR$ can be represented as a
  quadratic form $\varphi_{\ell} \colon S_1 \to \RR$. Since $\dim R_2 = 8$,
  the associated symmetric matrix of the corresponding quadratic form,
  relative to the ordered basis dual to $x_0, x_1, x_2, x_3$ in $S_1$, is a
  real matrix of the form
  \[
    \renewcommand{\arraystretch}{0.7}
    \renewcommand{\arraycolsep}{2pt}
    \begin{bmatrix*}
      -a_{1} & a_{2} & a_{3} & a_{4} \\
      a_{2} & a_{1} & a_{5} + a_{6} & a_{7} \\
      a_{3} & a_{5} + a_{6} & a_{5} & a_{8} \\
      a_{4} & a_{7} & a_{8} & a_{6} \\
    \end{bmatrix*}
    \, . \qedhere
  \]
\end{example}

We record some basic features of projective space. A closed point $p$ in
$\PP^n$ is an equivalence class consisting of nonzero elements in
$V \otimes_{\RR} \CC \cong \CC^{n+1}$ up to multiplication by a nonzero
complex number. Any element $\smash{\widehat{p}}$ in this equivalence class is
an affine representative of the point $p$. Complex conjugation on $\CC$
induces involutions on both the complex vector space $V \otimes_{\RR} \CC$ and
the closed points in $\PP^n$. A closed point $p$ is real if and only if it is
fixed under conjugation: $\overline{p} = p$. The set of real points in $\PP^n$
is denoted by $\PP^n(\RR)$ and a nonreal point on $\PP^n$ is any closed point
in the complement $\PP^n \setminus \PP^n(\RR)$. Every real point in $\PP^n$
admits real affine representatives and any two real affine representatives
coincide up to multiplication by a nonzero real number. Hence, there is a
bijection between $\PP^n(\RR)$ and the equivalence classes of nonzero elements
of $V$ up to multiplication by a real number. The set $\PP^n(\RR)$ is a
differentiable manifold endowed with a Euclidean topology induced by any norm
on $V$ under the canonical quotient map $V \setminus \{ 0 \} \to \PP^n(\RR)$.

We first analyze the geometric properties of the quadratic forms having corank
$1$. Although its proof is elementary, we were unable to find a suitable
reference. A subvariety $X$ in $\PP^n = \Proj(S)$ is \emph{nondegenerate} if
it is not contained in a hyperplane or, equivalently, if we have $R_1 = S_1$.

\begin{lemma}
    \label{l:many}
    Let $X$ be a nondegenerate real subvariety of $\PP^n$ with homogeneous
    coordinate ring $R$.
    \begin{compactenum}[\upshape (i)][3]
    \item The set
      $M \coloneq \bigl\{ \ell\colon R_2 \to \RR \mathrel{\big|} \textup{the
        quadratic form $\varphi_\ell$ has corank $1$} \bigr\}$ is a
      differentiable manifold of dimension $\dim R_2 -1$.
    \item For any quadratic form $\varphi$ in $M$, the tangent space
      $T_{\varphi}(M)$ is the linear subspace of $I_2^{\perp}$ consisting of
      the linear maps $\psi \colon S_2 \to \RR$ such that $\psi(g^2) = 0$ for
      all polynomials $g$ in the kernel of $\varphi$.
    \item The map $\Phi \colon M \to \PP^n(\RR)$, which sends a quadratic form
      $\varphi$ having corank $1$ to the equivalence class of its kernel, has
      a surjective differential at all points in $M$. In particular, the image
      of any Euclidean open set in $M$ has nonempty Euclidean interior in
      $\PP^n(\RR)$.
    \end{compactenum}
\end{lemma}

\begin{proof}
  Every quadratic polynomial $g$ in $S_2$ corresponds to a unique real
  symmetric matrix $\mathbf{A}$ where
  $g = \mathbf{x}^{\!\mathsf{T}} \,\mathbf{A} \, \mathbf{x}$ and
  $\mathbf{x} \coloneq \smash{
    \renewcommand{\arraystretch}{0.7}
    \renewcommand{\arraycolsep}{2pt}
    \begin{bmatrix*}
      x_0 & x_1 & \dotsb & x_n \\
    \end{bmatrix*}}^{\mathsf{T}}$.  
  For any two integers $j$ and $k$ satisfying
  $0 \leqslant j \leqslant k \leqslant n$, we write $a_{j,k}$ for the linear
  functional $a_{j,k} \colon S_2 \to \RR$ that determines the $(j,k)$-entry in
  the associated symmetric matrix $\mathbf{A}$. Let $Z$ be the closed affine
  subscheme of $S_2$ defined by the determinant of $\mathbf{A}$. The Jacobi
  formula
  $d \det(\mathbf{A}) = \operatorname{tr} \bigl(
  \operatorname{adj}(\mathbf{A}) \; d \mathbf{A} \bigr)$ expresses the
  derivative of this determinant in terms of the adjugate of $\mathbf{A}$ and
  the derivative of $\mathbf{A}$. Consider a point $\varphi$ in the
  determinantal hypersurface $Z$ corresponding to a quadratic form having
  corank $1$. As every symmetric matrix is orthogonally diagonalizable, we may
  choose coordinates so that the point $\varphi$ is represented by a diagonal
  matrix
  $\mathbf{Q} \coloneq \operatorname{diag}(\lambda_0, \lambda_1, \dotsc,
  \lambda_{n-1}, 0)$ where $\lambda_j \neq 0$ for all
  $0 \leqslant j \leqslant n-1$. Hence, the differential $d \det(\mathbf{A})$
  at the point $\varphi$ equals
  $\lambda_0 \, \lambda_1 \dotsb \lambda_{n-1} \; d a_{n,n} \neq 0$. It
  follows that $\varphi$ is a nonsingular point on $Z$.
    
  To establish that the point $\varphi$ is a nonsingular point on
  $Z \cap I_2^{\perp}$, we show that the tangent space $T_{\varphi}(Z)$ and
  the linear subspace $I_2^{\perp}$ intersect transversely, which proves that
  $M$ is nonsingular and $T_{\varphi}(M) = T_{\varphi}(Z) \cap I_2^{\perp}$.
  By hypothesis, the subvariety $X$ is nondegenerate, so the polynomial
  $x_n^2$ does not belong to $I_2$. Hence, the real vector space
  $\langle x_n^2 \rangle + I_2^{}$ has dimension $1 + \dim I_2$. Since
  $\smash{\bigl( T_{\varphi}(Z) \cap I_2^{\perp} \bigr)}^{\perp} =
  T_{\varphi}(Z)^{\perp} \cap I_2^{} = \langle x_n^2 \rangle + I_2^{}$, we see
  that $T_{\varphi}(Z)$ and $I_2^{\perp}$ intersect transversely. Because this
  calculation is invariant under orthogonal transformations, we deduce that
  parts~(i) and (ii) hold.

  To understand the map $\Phi \colon M \to \PP^n(\RR)$, recall that the
  adjugate matrix of $\mathbf{A}$ satisfies the equation
  $\mathbf{A} \, \operatorname{adj}(\mathbf{A}) = \det(\mathbf{A}) \,
  \mathbf{I}$. Hence, in a neighbourhood of the point $\varphi$ in $M$
  represented by the diagonal matrix $\mathbf{Q}$, the map $\Phi$ sends the
  associated symmetric matrix $\mathbf{A}$ to the equivalence class spanned by
  the last column of the adjugate matrix $\operatorname{adj}(\mathbf{A})$,
  whose entries are polynomials in the $a_{j,k}$. From this local description,
  we see that the map $\Phi \colon M \to \PP^n(\RR)$ is differentiable.

  Finally, we compute the differential $d \Phi$ at the point $\varphi$ in $M$.
  For all $0 \leqslant j \leqslant n$ and all $0 \leqslant k \leqslant n$, let
  $\mathbf{E}_{j,k}$ be the
  $\smash{\bigl( \!( n\!+\!1) \mathbin{\!\times\!} (n\!+\!1) \!\bigr)}$-matrix
  whose $(j,k)$-entry is $1$ and all other entries are $0$. By identifying
  points in dual space $S_2^*$ with their associated symmetric matrices, we
  see that the differentiable curves $\sigma_{j,k} \colon \RR \to S_2^*$
  defined by
  \[
    \sigma_{j,k}(t) \coloneq
    \begin{cases}
      \mathbf{Q} + t(\mathbf{E}_{j,k} + \mathbf{E}_{k,j})
      & \text{for all $0 \leqslant j \leqslant k \leqslant n-1$} \\[-2pt]
      \mathbf{Q} + t (\mathbf{E}_{j,n} + \mathbf{E}_{n,j}) + (t^2/\lambda_j) \mathbf{E}_{n,n}
      & \text{for all $0 \leqslant j \leqslant n-1$ and $k = n$} \\[-2pt]
    \end{cases}
  \]
  lie in the hypersurface $Z$ and their tangent directions at $t = 0$ span the
  tangent space $T_{\varphi}(Z)$. It follows that
  $\Phi(\varphi) = [0 \mathbin{:} 0 \mathbin{:} \dotsb \mathbin{:} 0
  \mathbin{:} 1 ]$ and
  \[
    \Phi \bigl( \!\sigma_{j,k}(t) \!\bigr) =
    \begin{cases}
      [ 0 \mathbin{:} 0 \mathbin{:} \dotsb \mathbin{:} 0 \mathbin{:} 1 ]
      & \text{for all $0 \leqslant j \leqslant k \leqslant n-1$} \\[-2pt]
      [ 0 \mathbin{:} 0 \mathbin{:} \dotsb \mathbin{:} 0 \mathbin{:} t
      \mathbin{:} 0 \mathbin{:} \dotsb \mathbin{:} 0 \mathbin{:} -
      \lambda_{j}]
      & \text{for all $0 \leqslant j \leqslant n-1$ and $k = n$.} \\[-2pt]
    \end{cases}
  \]
  Differentiating with respect to $t$ establishes that the differential
  $d \Phi$ surjects onto the tangent space
  $T_{\Phi(q)}\!\bigl( \PP^n(\RR) \!\bigr)$ and its kernel $K$ is spanned by
  forms that vanish at $x_{\!j} \, x_n$ for all $0 \leqslant j \leqslant n$.
  The map $\mu_{x_n} \colon R_1 \to R_2$ is injective, so the dimension of
  real vector space
  $\langle x_{\!j} \, x_n \mathrel{|} 0 \leqslant j \leqslant n \rangle +
  I_2^{}$ is equal to the sum of the dimensions of its summands. Since
  $\smash{\bigl( K \cap I_2^{\perp} \bigr)}^{\perp} = \langle x_{\!j} \, x_n
  \mathrel{|} 0 \leqslant j \leqslant n \rangle + I_2^{}$, we see that $K$
  and $I_2^{\perp}$ intersect transversely. Therefore, the differential of
  the map $\Phi \colon M \to \PP^n(\RR)$ is surjective at all points in $M$.
  The final assertion follows from the implicit function theorem.
\end{proof}

Our second lemma relates linear functionals to point evaluations. A closed
point in the subvariety $X$ is a closed point in $\PP^n$ at which the
polynomials in the homogeneous ideal $I$ vanish. The set of real points in
$X$ is denoted by $X(\RR)$. Given any closed point $p$ in $X$, any choice
$\smash{\widehat{p}}$ of affine representative defines a ring homomorphism
$\ev_p \colon R \to \CC$ by sending the coset $f$ in $R$ to the evaluation
$\smash{\widehat{f}(\widehat{p})}$, where $\smash{\widehat{f}}$ is a
polynomial in $S$ that maps to $f$ under the canonical surjection. Since the
point $p$ lies on $X$, the complex number $\smash{\widehat{f}(\widehat{p})}$
is independent of the choice of the polynomial $\widehat{f}$. The closed
point $p$ in $X$ determines the point evaluation $\ev_{p} \colon R \to \CC$
up to multiplication by a nonzero complex number. The affine representatives
of real points are always chosen to be real, so a point $p$ in $X(\RR)$
determines the map $\ev_{p} \colon R \to \RR$ up to multiplication by a
nonzero real number.

\begin{lemma} 
  \label{l:ciRa}
  Let $X$ be a real subvariety of $\PP^n$ with homogeneous coordinate ring
  $R$. For any linear functional $\ell \colon R_2 \to \RR$, there are
  nonnegative integers $r$ and $c$, real numbers $a_1, a_2, \dotsc, a_{r}$,
  real points $p_1, p_2, \dotsc, p_{r}$ in $X$, complex numbers
  $z_1, z_2, \dotsc, z_{c}$, and nonreal points $q_1, q_2, \dotsc, q_{c}$ in
  $X$ such that
  \[
    \ell
    = a_1^{} \, \ev_{p_1} \!+ \, a_2^{} \, \ev_{p_2} \!+ \dotsb
    \! + \, a_r^{} \, \ev_{p_r} \!+
    (z_1^{} \, \ev_{q_1} \!\!+ \, \overline{z_1^{}} \, \ev_{\overline{q_1}} ) +
    (z_2^{} \, \ev_{q_2} \!\!+ \, \overline{z_2^{}} \, \ev_{\overline{q_2}} )
    + \dotsb + 
    (z_c^{} \, \ev_{q_c} \!\!+ \, \overline{z_c^{}} \, \ev_{\overline{q_c}} ) \, .
  \]
\end{lemma}

\begin{proof}
  To start, we claim that any linear functional from $R_2$ to $\CC$ is a
  $\CC$-linear combination of point evaluations. The point evaluations span a
  linear subspace of the linear functionals $R_2^*$. Suppose that this linear
  subspace is contained in a hyperplane. It follows that the corresponding
  element $f$ in $R_2$ vanishes at every closed point in $X$. Since $X$ is a
  subvariety, its homogeneous ideal $I$ is radical, so any polynomial
  $\smash{\widehat{f}}$ in $S_2$ that maps to $f$ under the canonical
  surjection belongs to $I$. Hence, we have $f = 0$ and the linear subspace
  spanned by point evaluations is not contained in a nonzero hyperplane.

  It remains to describe the real-valued linear functionals
  $\ell \colon R_2 \to \RR$. The previous paragraph implies that there exists
  nonnegative integers $r$ and $c$, complex numbers $a_1, a_2, \dotsc, a_{r}$,
  real points $p_1, p_2, \dotsc, p_{r}$ in $X$, complex numbers
  $z_1, z_2, \dotsc, z_{c}$, and nonreal points $q_1, q_2, \dotsc, q_{c}$ in
  $X$ such that
  \[
    \ell
    = a_1 \, \ev_{p_1} + a_2 \, \ev_{p_2} + \dotsb + a_r \, \ev_{p_r}
    + z_1^{} \, \ev_{q_1} + z_2^{} \, \ev_{q_2} + \dotsb z_c^{} \, \ev_{q_c}
    \, .
  \]
  Since $\ell$ is real-valued and $p_{\!j} = \overline{p_{\!j}}$ for all
  $1 \leqslant j \leqslant r$, we see that $\ell = \overline{\ell}$ and
  \begin{align*}
    \ell
    = \tfrac{1}{2}(\ell + \overline{\ell}) 
    &= \operatorname{Re}(a_1^{}) \, \ev_{p_1} + \operatorname{Re}(a_2^{}) \,
      \ev_{p_2} + \dotsb + \operatorname{Re}(a_r^{}) \, \ev_{p_r} \\[-2pt]
    & \relphantom{WWW} + (z_1^{} \, \ev_{q_1} \!+ \overline{z_1^{}} \,
      \ev_{\overline{q_1}} ) + (z_2^{} \, \ev_{q_2} \!+ \overline{z_2^{}} \,
      \ev_{\overline{q_2}} ) + \dotsb (z_c^{} \, \ev_{q_c} \!+
      \overline{z_c^{}} \, \ev_{\overline{q_c}} ) \, . \qedhere
  \end{align*}
\end{proof}

Building on this lemma, we introduce two numerical invariants.

\begin{definition}
  \label{d:mcr}
  For a linear functional $\ell \colon R_2 \to \RR$, the
  \emph{conjugation-invariant length} $\rrank(\ell)$ is the minimum of the sum
  $r+c$ among all expressions for $\ell$ appearing in Lemma~\ref{l:ciRa}. The
  \emph{maximum typical conjugation-invariant length of $X$}, denoted by
  $\mcrk(X)$, is the smallest integer $k$ such that the subset of the linear
  functionals $\ell \colon R_2 \to \RR$ having conjugation-invariant length at
  most $k$ is dense in the Euclidean topology on $R_2^*$. Equivalently, the
  numerical invariant $\mcrk(X)$ is the smallest integer $k$ such that the
  subset of the linear functionals $\ell \colon R_2 \to \RR$ having
  conjugation-invariant length greater than $k$ does not contain a nonempty
  open neighbourhood in the Euclidean topology on $R^*_2$.
\end{definition}

The next proposition links the eigenvalues of a quadratic form to the
conjugation-invariant length. A set of points in $\PP^n$ is \emph{in linear
  general position} if they impose independent conditions on linear forms,
meaning that the set of linear forms vanishing at any $k \leqslant n+1$ of the
points is a linear subspace of codimension $k$.

\begin{proposition}
  Consider a linear function $\ell \colon R_2 \to \RR$ of the form
  \[
    \ell
    = a_1^{} \, \ev_{p_1} \!+ \, a_2^{} \, \ev_{p_2} \!+ \dotsb
    \! + \, a_r^{} \, \ev_{p_r} \!+ \,
    (z_1^{} \, \ev_{q_1} \!\!+ \, \overline{z_1^{}} \, \ev_{\overline{q_1}} ) +
    (z_2^{} \, \ev_{q_2} \!\!+ \, \overline{z_2^{}} \, \ev_{\overline{q_2}} )
    + \dotsb + 
    (z_c^{} \, \ev_{q_c} \!\!+ \, \overline{z_c^{}} \, \ev_{\overline{q_c}} ) \, ,
  \]
  where $a_1, a_2, \dotsc, a_r$ are real numbers, $p_1, p_2, \dotsc, p_r$ are
  real points in $\PP^n$, $z_1, z_2, \dotsc, z_c$ are complex numbers, and
  $q_1, q_2, \dotsc, q_c$ are nonreal points in $\PP^n$. Setting $r_{\!+}$ and
  $r_{\!-}$ to be the number of positive and negative numbers in the set
  $\{ a_1, a_2, \dotsc, a_r \}$ respectively, the corresponding quadratic form
  $\varphi_{\ell}$ has at most $r_{\!+} + c$ positive eigenvalues and
  $r_{\!-} + c$ negative eigenvalues. In particular, $\varphi_\ell$ has at
  most $\rrank(\ell)$ positive eigenvalues.
\end{proposition}

\begin{proof}
  For any real point $p$ in $\PP^n(\RR)$, the quadratic form corresponding to
  $\ev_{p} \colon R_2 \to \RR$ has rank one and one positive eigenvalue
  because $\ev_p(f^2) \geqslant 0$ for any $f$ in $R_1$. For any complex
  number $z$ and any nonreal point $q$ in $\PP^n$, we claim that the quadratic
  form corresponding to $z \, \ev_q + \overline{z} \, \ev_{\overline{q}}$ has
  rank at most two, and at most one positive and one negative eigenvalue.
  Since $g(\overline{q}) = \overline{g(q)}$ for any $g$ in $R_2$, we have
  $(\ev_q + \ev_{\overline{q}})(f^2) = 2 \operatorname{Re}\bigl(f^2(q)
  \kern-1.0pt \bigr)$ for any $f$ in $R_1$. Writing $q = a + \ii b$ for some
  real points $a$ and $b$ in $\PP^n$ and using that $f$ is linear, we see that
  $2\operatorname{Re} \kern-0.5pt \big(f^2(q) \kern-1.0pt \bigr) = 2 \bigl(
  f^2(a) - f^2(b) \kern-1.0pt \bigr)$. By the first part, we see that this
  quadratic form has rank at most two, and at most one positive and one
  negative eigenvalue. As $z \, \ev_q$ is the same as $\ev_{\sqrt{z} \, q}$ on
  any elements $g$ in $R_2$, the claim follows. Finally, we observe that the
  signature of a quadratic form is subadditive.
\end{proof}

%%%%%%%%%%%%%%%%%%%%%%%%%%%%%%%%%%%%%%%%%%%%%%%%%%%%%%%%%%%%%%%%%%%%%%%%%%%%%%
\section{Bounds for Curves}

\noindent
For real curves with no real points, this section bounds the maximum typical
conjugation-invariant length. Despite superficial similarities to Theorem~1 of
\cite{BT}, our new inequality is essentially independent and crucial to our
proof strategy. Being able to sharpen or generalize the inequality in
Theorem~\ref{t:mcBound} would ultimately translate into better degree bounds.

The next two lemmas provide topological tools for bounding the maximum typical
conjugation-invariant length on certain real subvarieties. The subvariety $X$
is \emph{linearly normal} if the canonical map
$H^0 \kern-1pt \bigl( \PP^n, \sO_{\PP^n}(1) \kern-1.5pt \bigr) \kern-1.0pt \to
H^0 \kern-1pt \bigl( X, \sO_{X}(1) \kern-1.5pt \bigr)$ is surjective.

\begin{lemma}
  \label{l:key}
  Assume that the real projective subvariety $X$ in $\PP^{n}$ is
  nondegenerate, linearly normal, and has positive dimension. Let $R$ be the
  homogeneous coordinate ring of $X$. For any linear subspace $W$ in $R_1$,
  the set of linear functionals $\ell\in R_2^*$ such that the corresponding
  quadratic form $\varphi_{\ell} \colon R_1 \to \RR$ satisfies the following
  three properties, is open and nonempty in the Zariski topology.
  \begin{compactenum}[\upshape (1)][3]
  \item The restriction of quadratic form $\varphi_{\ell}$ to the linear
    subspace $W$ has rank equal to $\dim W$.
  \item For any nonreal point $q$ in $X$, there exists a complex number $z$
    (which may depend on ${\ell}$) such that, for the linear functional
    $\ell' = z \, \ev_{q} + \overline{z} \, \ev_{\overline{q}}$, the quadratic
    form $\varphi'$ corresponding to ${\ell} + {\ell'}$ has corank $1$.
  \item The quadratic form $\varphi'$ is a smooth point on the differentiable
    manifold $M$ formed by the quadratic forms having corank $1$, and the
    linear map $\ell'$ is not an element of the tangent space
    $T_{\varphi'}(M)$.
  \end{compactenum}
\end{lemma}

Recall that Lemma~\ref{l:many}~(ii) identifies elements in $T_{\varphi'}(M)$
with appropriate linear maps.

\begin{proof}
  Fix a linear subspace $W$ in the $(n+1)$-dimensional real vector space $R_1$.
  \begin{compactenum}[\upshape (1)][3]
  \item A quadratic form has maximal rank if and only if its corresponding
    symmetric matrix is invertible. Hence, the locus of quadratic forms
    arising from linear functionals on $R_2$, whose restriction to $W$ has
    maximal rank, is open in the Zariski topology. We need to show that it is
    nonempty. Consider the incidence correspondence
    $\Xi \coloneq \bigl\{ (\ell, p) \in R_2^{*} \times \PP(W) \mathrel{\big|}
    \widehat{p} \in \Ker(\varphi_{\ell}) \bigr\}$. The fibre over the point
    $p$ in $\PP(W)$ consists of those linear functionals
    $\ell \colon R_2 \to \RR$ that annihilate the linear subspace
    $\bigl( \sum_{j=0}^{n} (\widehat{p})_{\!j} \, x_j \bigr) \cdot W$ in
    $R_2$. Since $X$ is nondegenerate and irreducible, multiplication by the
    linear form $\sum_{j=0}^{n} (\widehat{p})_{\!j} \, x_j$ is injective for
    any affine representative $\widehat{p}$. Hence, the dimension of the
    linear subspace
    $\smash{\bigl( \sum_{j=0}^{n} (\widehat{p})_{\!j} \, x_j \bigr) \cdot W}$
    equals $\dim W$ and the fibre over the point $p$ is a linear subspace of
    dimension $\dim R_2 - \dim W$. Combining the dimension of the fibres with
    the dimension of the base, the dimension of $\Xi$ is bounded above by
    $\dim R_2 - 1$. We deduce that the projection of the incidence
    correspondence $\Xi$ on first factor $R_2^*$ cannot be surjective.
    Therefore, a generic quadratic form $\varphi_{\ell}$ has maximal rank when
    restricted to $W$.
  \item By making a linear change of variables on the polynomial ring $S$, we
    may assume that $x_n(q) \neq 0$ and $x_0, x_1, \dotsc, x_{n-2}$ form a
    basis for the linear subspace of forms in $S_1$ vanishing at the nonreal
    points $q$ and $\overline{q}$. We may also assume that
    $[0 \mathbin{:} 0 \mathbin{:} \dotsb \mathbin{:} 0 \mathbin{:} \ii
    \mathbin{:} 1]$ is an affine representative of $q$. With respect to these
    coordinates, the symmetric
    $\bigl( \!(n \!+\! 1) \mathbin{\!\times\!} (n \!+\! 1) \!\bigr)$-matrix
    associated to $\varphi_{\ell'}$ is
    \[
      \renewcommand{\arraystretch}{0.9}
      \renewcommand{\arraycolsep}{4pt}
      \begin{bmatrix*}
        0 & 0 & \!\!\dotsb\!\! & 0 & 0 & 0 \\[-2pt]
        0 & 0 & \!\!\dotsb\!\! & 0 & 0 & 0 \\[-6pt]
        \vdots & \vdots & \!\!\ddots\!\! & \vdots & \vdots & \vdots \\[-2pt]
        0 & 0 & \!\!\dotsb\!\! & 0 & 0 & 0 \\[-2pt]        
        0 & 0 & \!\!\dotsb\!\! & 0 & - u & v \\[-2pt]
        0 & 0 & \!\!\dotsb\!\! & 0 & v & u \\[-1pt]
      \end{bmatrix*}
    \]
    where $z \coloneq \frac{1}{2}u - \ii \, v$ for some real numbers $u$ and
    $v$. Suppose that $\varphi_{\ell} \colon S_1 \to \RR$ is a quadratic form
    that has maximal rank when restricted to the linear subspace spanned by
    $x_0, x_1, \dotsc, x_{n-2}$. Hence, there exists an invertible
    $\smash{\bigl( \kern-1.5pt (n \!-\! 1) \mathbin{\!\times\!} (n \!-\!1)
      \kern-1.5pt \bigr)}$-matrix $\mathbf{A}$, a
    $\smash{\bigl( \kern-1.5pt ( n \!-\! 2) \mathbin{\!\times\!} 2
      \bigr)}$-matrix $\mathbf{B}$, and a $(2 \mathbin{\!\times\!} 2)$-matrix
    $\mathbf{C}$ such that the associated symmetric matrix has the block
    structure
    \[
      \renewcommand{\arraystretch}{0.9}
      \renewcommand{\arraycolsep}{4pt}
      \begin{bmatrix*}
        \mathbf{A} & \mathbf{B} \\
        \mathbf{B}^{\textsf{T}} & \mathbf{C} \\
      \end{bmatrix*} \, . 
    \]
    The Schur complement of the block $\mathbf{A}$ is the $(2 \mathbin{\!\times\!} 2)$-matrix
    \[
      \renewcommand{\arraystretch}{0.9}
      \renewcommand{\arraycolsep}{4pt}
      \begin{bmatrix*}
        \alpha & \beta \\
        \beta & \gamma \\
      \end{bmatrix*}
      \coloneq \mathbf{C} - \mathbf{B}^{\textsf{T}} \, \mathbf{A}^{-1} \,
      \mathbf{B} \, .
    \]
    Set $z \coloneq \frac{1}{2}\alpha + \ii \, \beta$ for the real numbers
    $\alpha$ and $\beta$. It follows that the analogous Schur complement of
    the symmetric matrix associated to
    $\varphi' \coloneq \varphi_{\ell} + \varphi_{\ell'}$ is the
    $(2 \mathbin{\!\times\!} 2)$-matrix
    \[
      \renewcommand{\arraystretch}{0.9}
      \renewcommand{\arraycolsep}{4pt}
      \begin{bmatrix*}
        0 & 0 \\
        0 & \alpha + \gamma \\
      \end{bmatrix*} 
    \]
    so the quadratic form $\varphi'$ does not have maximal rank. Moreover,
    $\varphi'$ has corank $1$ if and only if this Schur complement does not
    have trace zero. The locus of quadratic forms, for which the leading
    principal
    $\smash{\bigl( \! (n \!-\! 1) \mathbin{\!\times\!} (n \!-\! 1)
      \!\bigr)}$-submatrix is invertible and corresponding Schur complement
    has nonzero trace, is open in the Zariski topology. We need to show that
    this locus is nonempty.
    
    To accomplish this, we exhibit a quadratic form having corank $1$ such
    that the leading principal
    $\smash{\bigl( \kern-1.5pt (n \!-\! 1) \mathbin{\!\times\!} (n \!-\! 1)
      \kern-1.5pt \bigr)}$-submatrix of the associated symmetric matrix is
    invertible. Since $x_n$ is a general linear form, the Bertini
    Theorem~\cite{Jouanolou}*{Th\'{e}or\`{e}me~6.3} shows that the hyperplane
    section $Y \coloneq X \cap \variety(x_n)$ is reduced and nondegenerate.
    Moreover, as $X$ is linearly normal, the homogeneous coordinate ring of
    $Y$ is isomorphic to $R/\!\ideal{\smash{x_n}}$ in degrees at most $2$, so
    any quadratic form on $Y$ lifts to a quadratic form on $X$; see
    \cite{Zak}*{Lemma~2.9b}. When $X$ has dimension greater than $1$, the
    scheme $Y$ is irreducible and no product of linear forms vanishes on $Y$.
    Part~(1) implies that there exists a quadratic form on $Y$ whose
    restriction to any linear subspace has maximal rank. Hence, the leading
    principal
    $\smash{\bigl( \kern-1.5pt (n \! - \! 1) \mathbin{\!\times\!} (n \! -\! 1)
      \kern-1.5pt \bigr)}$-submatrix of the associated symmetric matrix is
    invertible, and this quadratic form on $Y$ lifts to quadratic form on $X$
    with the required properties. When $X$ has dimension $1$, the hyperplane
    section $Y$ is a reduced nondegenerate set of points in linearly general
    position; see~\cite{ACGH}*{p.~109}. Choose a minimal conjugation-invariant
    generating set
    $\{ p_1, p_2, \dotsc, p_r, q_{1}, \overline{q_1}, q_{2}, \overline{q_{2}},
    \dotsc, q_{c}, \overline{q_{c}}\}$ in $Y$ that spans $\PP^{n-1}$
    containing $r$ real points and $c$ pairs of conjugate complex points. The
    minimality implies that $r+c$ is either $n$ or $n+1$.
    For any general real numbers $a_1, a_2, \dotsc, a_r$ and any general
    complex numbers $z_1, z_2, \dotsc, z_c$, the quadratic form corresponding
    to the linear function
    \[
      a_1 \ev_{p_1} \!+ \, a_2 \ev_{p_2} \!+ \dotsb
      \! + \, a_r \ev_{p_r} \!+ \,
      (z_1 \ev_{q_1} \!\!+ \, \overline{z_1}\ev_{\overline{q_1}} ) +
      (z_2 \ev_{q_2} \!\!+ \, \overline{z_2} \ev_{\overline{q_2}} )
      + \dotsb + 
      (z_c \ev_{q_c} \!\!+ \, \overline{z}_c\ev_{\overline{q_c}} ) \, ,
    \]
    has maximal rank when restricted to the linear subspaces spanned by
    $x_0, x_1, \dotsc, x_{n-1}$ and $x_0, x_1, \dotsc, x_{n-2}$. When
    $r+c = n$, the general linear combination gives a quadratic form of rank
    $n$. The restriction to the linear subspace spanned by
    $x_0, x_1, \dotsc, x_{n-2}$ is the quadratic form corresponding to the
    evaluations at the projections of the points to the corresponding
    hyperplane $\variety(x_{n-1})$ in $\PP^{n-1}$. The projections are still
    in linearly general position and hence the general linear combination
    still has full rank. The case $r+c = n+1$ is analogous. It follows that
    this quadratic form on $Y$ lifts to quadratic form on $X$ with the
    required properties.
  \item Part~(2) establishes that the quadratic form $\varphi'$ has corank
    $1$, so Lemma~\ref{l:many}~(i) demonstrates that this quadratic form
    determines a point on the differentiable manifold $M$. Since the
    irreducibility of $X$ ensures that no nonzero linear form is a zerodivisor
    on $R$, Lemma~\ref{l:many}~(ii) shows that the point $\varphi'$ is
    nonsingular. Moreover, the tangent space $T_{\varphi'}(M)$ consists of
    those linear maps $\psi \colon S_2 \to \RR$ such that $\psi(g^2) = 0$ for
    all polynomials in the kernel of $\varphi'$. From the chosen coordinates
    in part~(2), we see that the linear map $\ell' \colon S_2 \to \RR$ does
    not lie in the tangent space $T_{\varphi'}(M)$ if and only if the real
    number $\alpha$, which is defined to be the $(1,1)$-entry in the Schur
    complement of the leading principal
    $\smash{\bigl( \! (n \!-\! 1) \mathbin{\!\times\!} (n \!-\! 1)
      \!\bigr)}$-submatrix, is nonzero. The locus of quadratic forms, for
    which the leading principal
    $\smash{\bigl( \! (n \!-\! 1) \mathbin{\!\times\!} (n \!-\! 1)
      \!\bigr)}$-submatrix is invertible and the $(1,1)$-entry in the
    corresponding Schur complement is nonzero, is open in the Zariski
    topology. It is nonempty because there exists a real number $u$ such that,
    for the linear functional
    $\ell'' \coloneq u (\ev_{q} + \ev_{\overline{q}})$, the quadratic form
    $(\varphi_{\ell} + \varphi_{\ell''}) + (\varphi_{\ell'} -
    \varphi_{\ell''})$ satisfies all three properties. \qedhere
  \end{compactenum}
\end{proof}

The second lemma brings the Euclidean topology into play.

\begin{lemma}
  \label{l:mKer}
  Assume that the real projective subvariety $X$ in $\PP^{n}$ is irreducible,
  linearly normal, and has positive dimension. Let $R$ be the homogeneous
  coordinate ring of $X$. Fix a nonreal point $q$ in $X$ and let
  $U \subset R_2^*$ be a nonempty Euclidean open set.
  \begin{compactenum}[\upshape (i)][2]
  \item There exist a linear functional $\ell$ in $U$ and a complex number $z$
    such that, for the linear functional
    $\ell_q \coloneq z \, \ev_{q} + \overline{z} \, \ev_{\overline{q}}$, the
    quadratic form corresponding to ${\ell} + {\ell_q}$ has corank $1$.
  \item There exists a Euclidean open set $U'$ containing the linear
    functional ${\ell}$ from part~(i) and a differentiable function
    $\lambda \colon U' \to \RR$ such that $U'$ is a dense subset of $U$ and,
    for all $\ell'$ in $U'$, the quadratic form corresponding to
    $\ell' + \lambda(\ell') \, {\ell_q}$ has corank $1$. Moreover, the locus
    of points in $\PP^n(\RR)$ determined by the kernels of quadratic forms
    corresponding to elements in $U'$ contains a Euclidean open subset.
  \end{compactenum}
\end{lemma}

\begin{proof} 
  For part~(i), Lemma~\ref{l:key} shows that the locus $U''$ of linear
  functionals ${\ell} \colon R_2 \to \RR$, such that there exists a complex
  number $z$ and the linear functional
  $\ell' \coloneq z \, \ev_{q} + \overline{z} \, \ev_{\overline{q}}$ for which
  the quadratic form $\varphi'$ corresponding to ${\ell} + {\ell'}$ has corank
  $1$ and the linear functional $\ell'$ is not an element in the tangent space
  $T_{\varphi'}(M)$, is nonempty and Zariski open. It follows that $U''$ has a
  nontrivial intersection with any nonempty Euclidean open set. Hence, we have
  $U'' \cap U \neq \varnothing$. Any quadratic form $\ell$ in $U'' \cap U$,
  together with the associated complex number $z$, proves part~(i).
    
  For part~(ii), the determinant of the symmetric matrix associated to
  $\varphi_{\ell} + t \, \varphi_{\ell'}$ is a polynomial in $\RR[t]$ having a
  simple root at $t = 1$, because the quadratic form
  $\varphi' \coloneq \varphi_{\ell} + \varphi_{\ell'}$ has corank $1$. By the
  Implicit Function Theorem, it follows that there exists a Euclidean open
  subset $U'$ with $\varphi_{\ell} \in U' \subseteq U$ and a differentiable
  function $\lambda \colon U' \to \RR$ such that, for all $\varphi$ in $U'$,
  the quadratic form $\varphi + \lambda(\varphi) \, \varphi_{\ell'}$ has
  corank $1$. Since the linear function $\ell'$ is not an element of the
  tangent space $T_{\varphi'}(M)$, the differential of this map is surjective
  at $\varphi_{\ell}$. Hence, this differential is an open map is some
  neighbourhood of the quadratic form $\varphi_{\ell}$. The locus where this
  fails is determined by the algebraic condition that the differential does
  not have full rank and is therefore lower-dimensional in $U'$.
\end{proof}

To bound the maximum typical conjugation-invariant length, we restrict our
attention to curves with no real points.

\begin{theorem}
  \label{t:mcBound}
  Assume that the real projective subvariety $X$ is nondegenerate,
  geometrically irreducible, and has dimension $1$. Let $R$ be the homogeneous
  coordinate ring of $X$. When $X$ is linearly normal and has no real points,
  we have the inequality
  $\mcrk(X) \leqslant 1+\left\lceil (\dim R_2-\dim R_1)/2\right\rceil$.
\end{theorem}

\begin{proof} 
  Let $\ell \colon R_2 \to \RR$ be a linear functional. We first reduce to
  points by identifying a suitable hyperplane section of the curve $X$. Pick a
  nonreal point $q$ in $X$. By Lemma~\ref{l:mKer}, there is a nonempty
  Euclidean open subset $U'\subset R_2^*$ containing $\ell$ in its closure and
  a differentiable function $\lambda \colon U' \to \RR$ such that for all
  $\ell'$ and $\ell''$ in $U'$, the quadratic form corresponding to
  $\ell'' + \lambda(\ell'') \, {\ell'}$ has corank $1$. Moreover, the locus of
  points in $\PP^{n}(\RR)$ determined by the kernels of quadratic forms
  corresponding to elements in $U'$ contains a Euclidean open subset. Consider
  the nonempty Euclidean open set $U'' \subseteq U'$ obtained by intersecting
  $U'$ with the nonempty Zariski open set of linear functionals such that a
  generator $h$ in $S_1$ of the kernel of the corresponding quadratic form is
  a nonzero divisor on $R$ and $X \cap \variety(h)$ is a reduced nondegenerate
  set of points in linearly general position. The Bertini
  Theorem~\cite{Jouanolou}*{Th\'{e}or\`{e}me~6.3} shows that a general
  hyperplane section of $X$ is reduced and nondegenerate, and the General
  Position Theorem~\cite{ACGH}*{p.~109} shows that the points are in linearly
  general position. Hence, the linear functional $\ell$ is in the closure of
  $U''$. It suffices to show that, for any $\ell''$ in $U''$, we have the
  desired inequality for $\rrank(\ell'')$.

  To produce the desired inequality, fix a general hyperplane $h$ and consider
  the hyperplane section $Y \coloneq X \cap \variety(h)$ of $X$. As $X$ is
  linearly normal, the homogeneous coordinate ring $T$ of $Y$ is isomorphic to
  $R/\!\ideal{h}$ in degrees at most $2$; see~\cite{Zak}*{Lemma~2.9b}. Hence,
  the ring $T$ is reduced in degree $2$, so every linear functional
  $\ell \colon T_2\otimes \CC \to \CC$ is a linear combination of point
  evaluations; compare with Lemma~\ref{l:ciRa}. Since $Y$ consists of points
  in linearly general position and contains no real points, the $\CC$-vector
  space $T_2 \otimes_{\RR} \CC$ has a generating set that is invariant under
  complex conjugation: $T_2 \otimes_{\RR} \CC$ is spanned by
  $\bigl\{ \ev_{q_i}, \ev_{\overline{q_i}} \mathrel{\big|} 1 \leqslant i
  \leqslant r \bigr\}$. It follows that the $\RR$-vector space $T_2$ is
  spanned by the set
  $\bigl\{ \ev_{q_i} + \ev_{\overline{q}_i} \mathrel{\big|} 1 \leqslant i
  \leqslant r \bigr\}$. Hence, the linear functional $\ell$ is a linear
  combination of $\smash{\lceil (\dim T_2)/2 \rceil}$ conjugate pairs. For all
  $\ell'' \in U'$, there exists
  $\ell' = z \ev_q + \overline{z} \ev_{\overline{q}}$ such that the quadratic
  form corresponding to ${\ell''} + \lambda(\ell'')\ell'$ has corank $1$, so
  we obtain
  \[ 
    \rrank(\ell'') 
    \leqslant 1 + \rrank\bigl( \ell''+\lambda(\ell'') \ell' \bigr)
    \leqslant 1 + \bigl\lceil (\dim T_2)/2 \bigr\rceil 
    = 1 + \bigl\lceil (\dim R_2 - \dim R_1)/2 \bigr\rceil \, . 
  \]
  Lastly, we conclude that
  $\mcrk(X) \leqslant 1 + \bigl\lceil (\dim R_2 - \dim R_1)/2 \bigr\rceil$
  because any $\ell\in R_2^*$ is the limit of linear functionals of
  conjugation-invariant length at most this bound.
\end{proof}

%%%%%%%%%%%%%%%%%%%%%%%%%%%%%%%%%%%%%%%%%%%%%%%%%%%%%%%%%%%%%%%%%%%%%%%%%%%%%%
\section{Nonnegative Multipliers on Surfaces}
\label{s:mainproof}

\noindent
This section presents the major technical result in the paper: we exhibit
cohomological conditions on a real surface that lead to certificates for
nonnegativity. Assume that the real projective variety $X$ is \emph{totally
  real}, meaning the set of real points in $X$ is Zariski dense in the set of
complex points or, equivalently, the variety has a nonsingular real point. A
divisor $D$ on $X$ is a Cartier divisor that is locally defined by a rational
function with \emph{real} coefficients. The divisor $D$ determines the
invertible sheaf or line bundle $\sO_{X}(D)$ on $X$.

A global section $f$ in
$\smash{H^0 \kern-1.0pt\bigl( X, \sO_{X}(2D) \kern-1.0pt \bigr) \kern-1.0pt}$
is \emph{nonnegative} if its sign at every real point in $X$ is not negative.
The sign is well-defined at a real point because $f$ is locally defined by a
rational function with real coefficients, the ratio of any two local
representatives is the square of an invertible section of $\sO_{X}$ evaluated
at the real point, and the square of any real number is nonnegative.
Similarly, a global section $f$ in
$\smash{H^0 \kern-1.0pt\bigl( X, \sO_{X}(2D) \kern-1.0pt \bigr) \kern-1.0pt}$
is a \emph{sum of squares} if there exists global sections
$h_1, h_2, \dotsc, h_r$ in
$\smash{H^0 \kern-1.0pt\bigl( X, \sO_{X}(D) \kern-1.0pt \bigr) \kern-1.0pt}$
such that $f = h_1^2 + h_2^2 + \dotsb + h_r^2$. Both of these properties make
sense for line bundles associated to an even divisor, i.e.,~the square of a
line bundle on $X$: $\sO_{X}(2D) = \sO_{X}(D) \otimes \sO_{X}(D)$.

Building on these concepts, we introduce the following terminology.

\begin{definition}
  \label{d:supp}
  A divisor $E$ \emph{supports multipliers} for a divisor $D$ if, for any
  nonnegative global section $f$ in
  $\smash{H^0 \kern-1.0pt\bigl( X, \sO_{X}(2D) \kern-1.0pt \bigr)
    \kern-1.0pt}$, there exists a nonzero global section $g$ in
  $\smash{H^0 \kern-1.0pt\bigl( X, \sO_{X}(2E) \kern-1.0pt \bigr)
    \kern-1.0pt}$ such that the product $g \, f$ in
  $\smash{H^0 \kern-1.0pt\bigl( X, \sO_{X}(2D + 2E) \kern-1.0pt \bigr)
    \kern-1.0pt}$ is a sum of squares.
\end{definition}

The next result gives effective conditions for a divisor on a real surface $X$
to support multipliers for another divisor. For any integer $i$ and any
divisor $D$ on $X$, set
$h^i (X, D) \coloneq \dim \smash{H^i \kern-1.0pt\bigl( X, \sO_{X}(D)
  \kern-1.0pt \bigr)}$.

\begin{theorem}
  \label{t:ineq}
  Assume that $X$ is a totally-real geometrically-integral projective surface.
  Let $D$ and $E$ be divisors on $X$ with $D$ free (the line bundle
  $\sO_{X}(D)$ is globally generated), $D+E$ very ample, and
  $\smash{H^0 \kern-1.0pt\bigl( X, \sO_{X}(E-D) \kern-1.0pt\bigr)\kern-1.5pt}
  = \smash{H^1 \kern-1.0pt\bigl( X, \sO_{X}(D+E) \kern-1.0pt\bigr)\kern-1.5pt}
  = \smash{H^1 \kern-1.0pt\bigl( X, \sO_{X}(2E) \kern-1.0pt\bigr)\kern-1.5pt}
  = 0$. The inequality
  \[
    h^0(X, D+E) 
    > 1 + \left\lceil
      \frac{h^0(X, 2D+2E) - h^0(X, 2E) - h^0(X, D+E) - h^1(X, E-D)}%
      {2} \right\rceil
  \]
  implies that the divisor $E$ supports multipliers for the divisor $D$.
\end{theorem}

\begin{proof}[Proof by contrapositive]
  Suppose that the divisor $E$ does not support multipliers for the divisor
  $D$. We first identify a special witness for this failure to support
  multipliers. By definition, there exists a nonnegative global section $f$ in
  $\smash{H^0 \kern-1.0pt \bigl( X, \sO_{X}(2D) \kern-1.0pt \bigr)}$ such
  that, for any nonzero global section $g$ in
  $\smash{H^0 \kern-1.0pt \bigl( X, \sO_{X}(2E) \kern-1.0pt \bigr)}$, the
  product $g \, f$ in
  $\smash{H^0 \kern-1.0pt \bigl( X, \sO_{X}(2D + 2E) \kern-1.0pt
    \bigr)\kern-1.0pt}$ is not a sum of squares. Since $X$ is totally real and
  geometrically integral, the cone $\Sigma_{2D+2E}$ formed by the sums of
  squares in
  $\smash{H^0 \kern-1.0pt \bigl( X, \sO_{X}(2D + 2E) \kern-1.0pt
    \bigr)\kern-1.0pt}$ is pointed (closed in the Euclidean topology and
  contains no lines); see~\cite{BSV19}*{Proposition~2.5}. We deduce that the
  linear space
  $f \cdot \smash{H^0 \kern-1.0pt \bigl( X, \sO_{X}(2E) \kern-1.0pt \bigr)
    \kern-1.0pt}$ and the cone $\Sigma_{2D+2E}$ are well-separated: there is a
  linear functional that is positive on the nonzero elements in the linear
  space and negative on the nonzero elements in the cone. Continuity implies
  that this property also holds for any global section $f'$ sufficiently close
  to $f$ in the Euclidean norm. Since the base locus of the divisor $D$ is
  empty, there exist global sections $h_1, h_2, \dotsc, h_k$ in
  $\smash{H^0 \kern-1.0pt \bigl( X, \sO_{X}(D) \kern-1.0pt \bigr)}$ that have
  no common zero on $X$. Hence, for any sufficiently small positive real
  number $\epsilon$, the global section
  $f' \coloneq f + \epsilon (h_1^2 + h_2^{2} + \dotsb + h_k^2)$ is positive
  and
  $f' \cdot \smash{H^0 \kern-1.0pt \bigl( X, \sO_{X}(2E) \kern-1.0pt \bigr)}
  \cap \Sigma_{2D+2E} = \{ 0 \}$. It follows that the set of positive global
  sections $f'$ in
  $\smash{H^0 \kern-1.0pt \bigl( X, \sO_{X}(2D) \kern-1.0pt \bigr)}$ such that
  the linear space
  $f' \cdot \smash{H^0 \kern-1.0pt \bigl( X, \sO_{X}(2E) \kern-1.0pt \bigr)}$
  and the cone $\Sigma_{2D+2E}$ are well-separated has nonempty interior in
  the Euclidean topology on
  $\smash{H^0 \kern-1.0pt \bigl( X, \sO_{X}(2D) \kern-1.0pt \bigr)}$; compare
  with \cite{BSV19}*{Theorem~3.1}. As a consequence, this nonempty Euclidean
  open set must intersect the Zariski open set of global sections $f''$ in
  $\smash{H^0 \kern-1.0pt \bigl( X, \sO_{X}(2D) \kern-1.0pt\bigr)}$ for which
  the curve $Y$ on $X$ defined by the vanishing of $f''$ is reduced and
  geometrically integral. Our assumption that $D$ is free, combined with the
  Bertini Theorem~\cite{Jouanolou}*{Th\'{e}or\`{e}me~6.3}, ensure that this
  Zariski open set is nonempty. Therefore, we may assume that the global
  section $f$ in
  $\smash{H^0 \kern-1.0pt \bigl( X, \sO_{X}(2D) \kern-1.0pt \bigr)}$ is
  nonnegative and that the associated curve $Y = \variety(f)$ on $X$ is
  geometrically integral and contains no real points.

  We now show that the curve $Y$ in $X$ is equipped with some special linear
  functionals. Consider the section ring
  $\smash{\widehat{R} \coloneq \bigoplus_{n \in \NN} H^0 \kern-1.0pt\bigl( X,
    \sO_{X}(nD+nE) \kern-1.0pt\bigr) \kern-1.0pt}$ of the surface $X$. As an
  algebraic counterpart to $Y$, let $T$ be the quotient ring of $\widehat{R}$
  by the ideal generated by the linear subspace
  $f \cdot \smash{H^0 \kern-1.0pt \bigl( X, \sO_{X}(2E) \kern-1.0pt \bigr)}$
  in
  $\widehat{R}_2 = \smash{H^0 \kern-1.0pt \bigl( X, \sO_{X}(2D+2E) \kern-1.0pt
    \bigr)}$. Since
  $f \cdot \smash{H^0 \kern-1.0pt \bigl( X, \sO_{X}(2E) \kern-1.0pt \bigr)}
  \cap \Sigma_{2D+2E} = \{ 0 \}$, the image of the cone $\Sigma_{2D+2E}$ in
  $T_2$ is pointed. Hence, there exists a Euclidean open subset of linear
  functionals from $T_2$ to $\RR$ that are positive on the nonzero squares of
  elements from
  $T_1 = \widehat{R}_1 = \smash{H^0 \kern-1.0pt \bigl( X, \sO_{X}(D+E)
    \kern-1.0pt \bigr)}$. As a second variant of $Y$, let $\widehat{Y}$ be the
  image of $Y$ in
  $\PP \bigl( \smash{H^0 \kern-1.0pt\bigl( X, \sO_{X}(D+E) \kern-1.0pt
    \bigr)}^* \kern-1.0pt \bigr)$ under the morphism determined by the
  complete linear series of the line bundle $\sO_{Y}(D+E)$. By hypothesis, the
  divisor $D+E$ is very ample on $X$, so its restriction to the subvariety $Y$
  is also very ample. It follows that $\widehat{Y}$ is isomorphic to $Y$ and
  it contains no real points. In addition, the projective curve $\widehat{Y}$
  is linearly normal and its homogeneous coordinate ring $\widehat{B}$ is the
  subalgebra of the section ring
  $\smash{\widehat{T} \coloneq \bigoplus_{n \in \NN} H^0 \kern-1.0pt \bigl( Y,
    \sO_{Y}(nD+nE) \kern-1.0pt \bigr)}$ generated by
  $\smash{H^0 \kern-1.0pt \bigl( Y, \sO_{X}(D+E) \kern-1.0pt \bigr)}$.
  Tensoring the short exact sequence
  \[
    \begin{tikzcd}[column sep = 2.0em]
      0 \ar[r]
      & \sO_{X}(-2D) \ar[r]
      & \sO_{X} \ar[r]
      & \sO_{Y} \ar[r]
      & 0 \, 
    \end{tikzcd}
  \]
  of coherent sheaves with the line bundle $\sO_{X}(D+E)$ and the line bundle
  $\sO_{X}(2D+2E)$, we obtain the following two short exact sequences in
  cohomology
  \[
    \begin{tikzcd}[row sep = -5pt, column sep = 2.0em]
      0 \ar[r]
      & H^0 \kern-1.0pt\bigl( X, \sO_{X}(D+E) \kern-1.0pt\bigr) \kern-1.0pt \ar[r]
      & H^0 \kern-1.0pt\bigl( X, \sO_{Y}(D+E) \kern-1.0pt\bigr) \kern-1.0pt \ar[r]
      & H^1 \kern-1.0pt\bigl( X, \sO_{X}(E-D) \kern-1.0pt\bigr) \kern-1.0pt \ar[r]
      & 0 \, \phantom{,} \\
      0 \ar[r]
      & H^0 \kern-1.0pt\bigl( X, \sO_{X}(2E) \kern-1.0pt\bigr) \kern-1.0pt \ar[r]
      & H^0 \kern-1.0pt\bigl( X, \sO_{X}(2D+2E) \kern-1.0pt\bigr) \kern-1.0pt \ar[r]
      & H^0 \kern-1.0pt\bigl( X, \sO_{Y}(2D+2E) \kern-1.0pt\bigr) \kern-1.0pt \ar[r]
      & 0 \, ,
    \end{tikzcd}
  \]
  because
  $\smash{H^0 \kern-1.0pt \bigl( X, \sO_{X}(E-D) \kern-1.0pt \bigr)} =
  \smash{H^1 \kern-1.0pt \bigl( X, \sO_{X}(D+E) \kern-1.0pt \bigr)} =
  \smash{H^1 \kern-1.0pt \bigl( X, \sO_{X}(2E) \kern-1.0pt \bigr)} = 0$. We
  deduce that
  \begin{align*}
    \dim \widehat{T}_1
    = h^0(Y, D+E) 
    &= h^0(X, D+E) + h^1(X,E-D)
      \geqslant \dim T_1 \, , \\[-2pt]
    \dim \widehat{T}_2
    = h^0(Y, 2D+2E) 
    &= h^0(X, 2D+2E) - h^0(X,2E)
      = \dim T_{2} \, ,
  \end{align*}
  and $T_2 \cong \widehat{T}_2$. This canonical isomorphism implies that there
  exists a Euclidean open subset of linear functionals from $\widehat{T}_2$ to
  $\RR$ that are positive on nonzero squares of the elements from
  $T_1 \subseteq \widehat{T}_1$. From the inclusion
  $\widehat{B} \subseteq \widehat{T}$, we conclude that there exists a
  nonempty Euclidean open set $U$ of linear functionals from $\widehat{B}_2$
  to $\RR$ that are positive on nonzero squares of the elements from
  $T_1 \subseteq \widehat{T}_1 = \widehat{B}_1$.

  To complete the proof, we produce the appropriate inequality. Applied to the
  projective curve $\widehat{Y}$, Theorem~\ref{t:mcBound} shows that there
  exists a linear functional $\ell \in U$ such that
  \[
    \rrank(\ell)
    \leqslant \mcrk(X)  
    \leqslant 1 + \bigl\lceil ( \dim \widehat{B}_2 - \dim \widehat{B}_1 )/2 \bigr\rceil 
    \leqslant 1 + \bigl\lceil ( \dim \widehat{T}_2 - \dim \widehat{T}_1 )/2 \bigr\rceil
    \, . 
  \]
  By construction, the linear functional $\ell \colon \widehat{B}_2 \to \RR$
  is positive on the nonzero squares of elements from
  $T_1 \subseteq \widehat{B}_1$, so the associated quadratic form
  $\varphi_{\ell}$ is positive definite on $T_1$. Since the dimension of $T_1$
  is bounded above by the number $\rrank(\ell)$ of positive eigenvalues of the
  form, we have the inequality
  \begin{align*}
    h^0(X, D+E) 
    = \dim T_1
    &\leqslant 1 + \bigl\lceil ( \dim \widehat{T}_2 - \dim \widehat{T}_1 )/2 \bigr\rceil %\\
    % &= 1 + \left\lceil \bigl( h^0(X, 2D+2E) - h^0(X, 2E) - h^0(X, D+E) -
    %   h^1(X, E-D) \bigr) / 2 \right\rceil
        \, .  \qedhere
  \end{align*}
\end{proof}

When the divisors $D$ and $E$ are sufficiently positive, the inequality in
Theorem~\ref{t:ineq} can be rephrased in terms of Euler characteristics. For
any integer $m$, the \emph{Euler characteristic} the divisor $m \, D$ on $X$
is the integer $\chi(m \, D) \coloneq \sum_{i} (-1)^{i} \, h^{i}(X, m \, D)$.

\begin{corollary}
  \label{c:simp}
  Assume that $X$ is a totally-real geometrically-integral projective surface.
  Let $D$ and $E$ be divisors on $X$ with $D$ free, $D+E$ very ample and
  $h^0(X,E-D)=0$. When $h^i(X, m \, E) = 0$ and $h^i(X, mD+ mE) = 0$ for any
  positive integers $i$ and $m$, the inequality
  \[  
    \chi(2E) + h^1(X, E-D) > \chi(-D-E) 
  \]  
  implies that the divisor $E$ supports multipliers for the divisor $D$. When
  $X$ is nonsingular and $K_X$ is its canonical divisor, this inequality is
  equivalent to $h^0(X, 2E) + h^1(X, E-D) > h^0(X, K_X+D+E)$.
\end{corollary}
    
\begin{proof} 
  For the first part, it suffices to prove that the inequality in
  Theorem~\ref{t:ineq} follows from the first inequality. Because $X$ is a
  surface, the Riemann--Roch Theorem~\cite{Bea}*{Theorem~I.12} shows that
  $\chi(mE)$ and $\chi(mD+ mE)$ are quadratic polynomials in $m$ with
  half-integer coefficients whose constant terms equal $1$. Setting
  $\chi(E) = 1 + b_1m + b_2m^2$ and $\chi( mD+ mE) = 1 + a_1m + a_2m^2$ for
  some coefficients $b_1, b_2, a_1, a_2$ in $\ZZ\bigl[\tfrac{1}{2} \bigr]$,
  the inequality $\chi(2E) + h^1(X, E-D) > \chi(-D-E)$ becomes
  \begin{align*}
    && 1+2b_1 + 4b_2 + h^1(X, E-D) 
    &> 1 - a_1 + a_2 \\
    % & \Leftrightarrow
    % & 2a_1 + 2a_2 
    % &> 1 + a_1 + 3a_2 - h^1(X, E-D) - ( 1+2b_1 + 4b_2 ) \\
    & \Leftrightarrow
    & a_1 + a_2 
      &> \frac{2 + a_1 +  a_2 - h^1(X, E-D)}{2} + a_2 - (1 + b_1 + 2b_2) \, . 
  \end{align*}
  Since $\chi(D+E) = 1+a_1+a_2$ is an integer, the left side of this last
  inequality is an integer and the right side is either an integer or
  half-integer. For any integer $k$, the inequality
  $\tfrac{k}{2} \geqslant \bigl\lceil \tfrac{k}{2} - \tfrac{1}{2} \bigr\rceil$
  gives
  \begin{align*}
    \chi(D+E) -1 = a_1 +  a_2
    &> \left\lceil \frac{1 + a_1 +  a_2 - h^1(X, E-D)}{2} + a_2 - (1 + b_1 +
      2b_2) \right\rceil \\
    % &= \left\lceil \frac{(1+2a_1 + 4a_2) - (1 + 2b_1 +4b_2) - (1 + a_1 +
    %   a_2) - h^1(X, E-D)}{2}
    % \right\rceil \\
    &= \left\lceil \frac{\chi(2D + 2E) - \chi(2E) - \chi(D+E) - h^{1}(E-D)}{2}
      \right\rceil \, . 
  \end{align*}
  For any positive integers $i$ and $m$, the hypothesis that $h^i(X, mE) = 0$
  and $h^i(X, mD + mE) = 0$ implies that $\chi(X, mE) = h^0(X, mE)$ and
  $\chi( mD + mE) = h^0( mD + mE)$. Hence, we obtain the desired inequality
  \[
    h^0(X, D+E) 
    > 1 + \left\lceil
      \frac{h^0(X, 2D + 2E) - h^0(X, 2E) - h^0(X, D+E) - h^1(X, E-D)}%
      {2} \right\rceil \, .
  \]

  For the second part, the surface $X$ is nonsingular. Serre
  Duality~\cite{Bea}*{Theorem~I.11} shows that
  $\chi(-D-E) = \chi(K_{X} + D + E)$. Since the divisor $D+E$ is very ample
  (and thereby ample), the Kodaira Vanishing
  Theorem~\cite{PAG}*{Theorem~4.2.1} asserts that $h^i(X, K_{X}+D+E) = 0$ for
  all positive integers $i$. We conclude that
  $\chi(K_{X} + D + E) = h^0(X, K_{X}+D+E)$.
\end{proof}

In some situations, Theorem~\ref{t:ineq} can also be used inductively. A
variety is \emph{strongly totally-real} if the real points of every nonempty
Zariski open set are dense in the Euclidean topology on $X$. Equivalently,
every real point of $X$ is a \emph{central point}: it lies in the Euclidean
closure of the nonsingular real points; see~\cite{BCR}*{\S7.6}. For instance,
this holds whenever the nonsingular real points of $X$ are dense in the
Euclidean topology, so it holds for nonsingular surfaces. On any strongly
totally-real surface, a multiplier $g$ in
$H^0\bigl(X, \sO_{X}(2E) \kern-1.0pt \bigr)$ constructed via
Theorem~\ref{t:ineq} is necessarily a nonnegative global section, which allows
for the repeated applications of Theorem~\ref{t:ineq}. To formalize this idea,
we generalize Definition~\ref{d:supp}.

\begin{definition}
  \label{d:trans}
  Let $D$ and $E$ be divisors on $X$. For any positive integer $t$, a sequence
  $(D_0, D_1, \dotsc, D_t)$ of divisors on $X$ is a \emph{$t$-step transfer}
  from $D$ to $E$ if $D = D_{0}$, $E = D_{t}$, and the divisor $D_{i}$
  supports multipliers for the divisor $D_{i-1}$ for all
  $1 \leqslant i \leqslant t$.
\end{definition}

Given a $t$-step transfer from $D$ to $E$, the next corollary shows that the
problem of deciding whether a global section in
$\smash{H^0 \kern-1.0pt \bigl( X, \sO_{X}(2D)\kern-1.0pt\bigr) \kern-1.0pt}$
is nonnegative reduces to solving $t$ semidefinite programming problems and
deciding whether a global section in
$\smash{H^0 \kern-1.0pt \bigl( X, \sO_{X}(2E) \kern-1.0pt\bigr) \kern-1.0pt}$
is nonnegative.

\begin{corollary}
  \label{c:recursion}
  Assume that $X$ is strongly-totally-real geometrically-integral projective
  surface. Let $(D_0, D_1, \dotsc, D_t)$ be a $t$-step transfer on $X$ from a
  divisor $D$ to a divisor $E$. A global section $f_0$ in
  $\smash{H^0 \kern-1.0pt\bigl( X, \sO_{X}(2D) \kern-1.0pt \bigr)
    \kern-1.0pt}$ is nonnegative if and only if, for any integer $i$
  satisfying $1 \leqslant i \leqslant t$, there exists a nonnegative global
  section $f_i$ in
  $\smash{H^0 \kern-1.0pt \bigl( X, \sO_{X}(2D_i) \kern-1.0pt \bigr)
    \kern-1.0pt}$ and a sums-of-squares $h_{i}$ in
  $\smash{H^0 \kern-1.0pt \bigl( X, \sO_{X}(2D_{i-1}+2D_{i}) \kern-1.0pt
    \bigr)\kern-1.0pt}$ such that $f_{i-1} f_{i} = h_{i}$. Moreover, if $f_t$
  is a sum of squares, then both the multiplier $f_t \cdot f_{t-1} \dotsb f_1$
  in $H^0 \kern-1.0pt \bigl(X,\sO_X(\sum_{i=1}^t D_i) \kern-1.0pt \bigr)$ and
  the product $f_t \cdot f_{t-1} \dotsb f_1 \cdot f_0$ in
  $H^0 \kern-1.0pt \bigl( X, \sO_X(\sum_{i=0}^t D_i) \kern-1.0pt \bigr)$ are
  sums of squares.
\end{corollary}

\begin{proof} 
  By definition, the divisor $D_{i}$ supports multipliers on the divisor
  $D_{i-1}$ for any integer $i$ satisfying $1 \leqslant i \leqslant t$. Hence,
  for any nonnegative global section $f_{i-1}$ in
  $\smash{H^0 \kern-1.0pt \bigl( X, \sO_{X}(2D_{i-1}) \kern-1.0pt \bigr)}$,
  there exists a global section $f_{i}$ in
  $\smash{H^0 \kern-1.0pt \bigl( X, \sO_{X}(2D_{i}) \kern-1.0pt \bigr)}$ and a
  sum-of-squares $h_i$ in
  $\smash{H^0 \kern-1.0pt \bigl( X, \sO_{X}(2D_{i-1}+2D_{i}) \kern-1.0pt
    \bigr) \kern-1.0pt}$ such that $f_{i-1} f_{i} = h_{i}$. Since $X$ is
  strongly totally-real, this equality establishes that $f_{i}$ is also
  nonnegative. As this equality holds for all $1 \leqslant i \leqslant t$, the
  nonnegativity of $f_0$ implies that nonnegativity of $f_t$. Conversely, when
  $f_t$ is a nonnegative global section and the equality
  $f_{i-1} f_{i} = h_{i}$ holds for all $1 \leqslant i \leqslant t$, we
  successively deduce that $f_{t-1}, f_{t-2}, \dotsc, f_{1}, f_{0}$ are all
  nonnegative.

  Now assume that $f_t$ in
  $H^0 \kern-1.0pt \bigl(X, \sO_X(D_t) \kern-1.0pt \bigr)$ is a sum of
  squares. To show the last claim, we assume that $t \geqslant 2$ and even.
  The odd case follows similarly. Since products of sums of squares are again
  sums of squares, we see that both the multiplier
  $f_{t}f_{t-1}f_{t-2} \dotsb f_1 f_0 = (f_{t} f_{t-1}) \dotsb (f_4 f_3) (f_2
  f_1)$ and the product
  $f_{t}f_{t-1}f_{t-2} \dotsb f_1 f_0 = f_{t} (f_{t-1} f_{t-2}) \dotsb (f_3
  f_2) (f_1 f_0)$ are sums of squares.
\end{proof}

We end this section by clarifying the difference between the nonnegativity of
global sections and the nonnegativity of homogeneous forms. For a free divisor
$D$ on $X$, consider the morphism
$\nu \colon X \to \PP \kern-1.0pt \bigl( H^0 \kern-1.0pt \bigl( X, \sO_{X}(D)
\kern-1.0pt \bigr)^{\!*} \kern-0.5pt \bigr)$ determined by the complete linear
series of the line bundle $\sO_{X}(D)$. Let $X' \coloneq \nu(X)$ be the image
subvariety and let $R'$ be its homogeneous coordinate ring. The pullback under
$\nu$ of a homogeneous element in $R'_{2}$ that is nonnegative on $X'$ is a
nonnegative global section in
$\smash{H^0 \kern-1.0pt \bigl( X, \sO_{X}(2D)\kern-1.0pt \bigr)}$. However,
there could be nonnegative global sections in
$\smash{H^0 \kern-1.0pt \bigl( X, \sO_{X}(2D)\kern-1.0pt \bigr)}$ that do not
correspond to nonnegative elements in $R'_{2}$. Indeed, the definitions of
nonnegativity on $X$ and $X'$ may not coincide because the target $X'$ could
have real points that are not images under $\nu$ of the real points in $X$. To
avoid this disparity, we offer the following definition.

\begin{definition}
  A morphism $\mu \colon X \to X'$ between real varieties is \emph{strongly
    dominant} if the set of real points in $X'$ with the Euclidean topology
  have the image of the real points in $X$ as a dense subset. A free divisor
  $D$ on $X$ is \emph{strongly dominant} if the morphism from $X$ to its image
  subvariety under the associated complete linear series is strongly dominant.
\end{definition}

By design, the nonnegativity of a homogeneous element in $R'_2$ is equivalent
to the nonnegativity of a global section in
$\smash{H^0 \kern-1.0pt \bigl( X, \sO_{X}(2D) \kern-1.0pt \bigr)}$ whenever
the free divisor $D$ is strongly dominant. For the sake of completeness, we
exhibit a free divisor on a smooth surface that is not strongly dominant.

\begin{example}
  \label{e:Q31} 
  Consider the surface $Q^{3,1}$ in $\PP^3$ defined by
  $x_0^2 - x_1^2 - x_2^2 - x_3^2$ in $\RR[x_0,x_1,x_2,x_3]$. In the affine
  open subset defined by $x_0 \neq 0$, the real points in $Q^{3,1}$ form a
  sphere $S^2$. Choose real line $L$ disjoint from this sphere. It follows
  that the line $L$ intersects the subvariety $Q^{3,1}$ in a pair of complex
  conjugate points $q$ and $\overline{q}$. Let $Q^{3,1}(0,2)$ to be the
  blow-up of $Q^{3,1}$ at these two points. The real points of $Q^{3,1}(0,2)$
  are in bijection with the real points of $Q^{3,1}$ under the canonical
  morphism $\pi \colon Q^{3,1}(0,2) \to Q^{3,1}$. Write $E_1$ and $E_2$ for
  the exceptional divisors on $Q^{3,1}(0,2)$ over the points $q$ and
  $\overline{q}$ respectively, and set $H$ to be the pullback to
  $Q^{3,1}(0,2)$ of the hyperplane class in $\PP^3$ to $Q^{3,1}$. The complete
  linear system of the real divisor $H-E_1-E_2$ corresponds to a pencil of
  hyperplanes in $\PP^3$ containing the line $L$. The corresponding morphism
  $\nu \colon Q^{3,1}(0,2) \to \PP^1$ maps the real points in $Q^{3,1}(0,2)$
  to the closed interval consisting of those hyperplanes which intersect the
  sphere. Since this interval is not dense in the Euclidean topology on the
  real points in $\PP^1$, we conclude that the divisor $H-E_1-E_2$ is not
  strongly dominant.
\end{example}

%%%%%%%%%%%%%%%%%%%%%%%%%%%%%%%%%%%%%%%%%%%%%%%%%%%%%%%%%%%%%%%%%%%%%%%%%%%%%%
\section{Applications to Toric Surfaces}
\label{s:toric}

\noindent 
In this section, we apply Theorem~\ref{t:ineq} to the construction of
nonnegativity certificates on toric surfaces.

A binary Laurent polynomial is an expression
$f=\sum_{(a,b)\in S} c_{(a,b)}x^ay^b$ where $S\subseteq \ZZ^2$ is a given
finite set of exponents and $c_{(a,b)}$ are real numbers. If
$C\subseteq \RR^2$ is any subset, we say that $f$ has \emph{monomial support
  on $C$} if $S \subseteq C$ and that $f$ is nonnegative if
$f(\alpha,\beta) \geqslant 0$ for every nonzero real numbers $\alpha, \beta$
(i.e.{} for $(\alpha,\beta)\in (\RR^*)^2$ in the real points of the
$2$-dimensional algebraic torus). In this section, we study the problem of
certifying the nonnegativity of binary Laurent polynomials with monomial
support on $2P$ where $P$ is a given lattice polygon.

The main result of this section is Theorem~\ref{t:>tor} which provides a
combinatorial criterion on another lattice polygon $Q$ so that $2Q$ supports
multipliers for all nonnegative Laurent polynomials $f$ with monomial support
on $2P$. To state our main result we introduce the following terminology: If
$A \subseteq \RR^2$ then the number of \emph{reduced connected components of
  $A$} is one less than the number of connected components of $A$ and a
\emph{lattice translate of $A$} is a set of the form $A+m$ for $m \in \ZZ^2$.
We write $\#A$ for the number of lattice points contained in $A$ and
$A^{\circ}$ for the interior of $A$.

\begin{theorem}
  \label{t:>tor}
  Assume that $P$ and $Q$ are convex lattice polygons such that no lattice
  translate of $P$ is contained in $Q$. Let $h$ be the total number of reduced
  connected components of the set differences $P\setminus Q'$ as $Q'$ ranges
  over all lattice translates of $Q$. The inequality
  \[
    \#(2Q)+ h > \# \bigl( \kern-1.0pt (P+Q)^{\circ} \kern-1.0pt \bigr)
  \]
  implies that $Q$ supports multipliers for $P$ (i.e. for every nonnegative
  Laurent polynomial $f$ with monomial support in $2P$ there exists a Laurent
  polynomial $g$ with monomial support in $2Q$ such that $fg$ is a sum of
  squares of Laurent polynomials with monomial support in $2(P+Q)$).
\end{theorem}

% The proof of Theorem~\ref{t:>tor} consists of two steps. First, using the
% polygons $P$ and $Q$ we compactify the torus into a toric variety $X$,
% defined by the normal fan of $P+Q$. The polygons $P$ and $Q$ determine
% Cartier divisors on $X$ and Laurent polynomials with supports in sums of
% $P,Q$ become sections of the corresponding line bundles. We then apply
% Theorem~\ref{t:ineq} to the toric surface $X$. In order to compute the
% relevant cohomology groups we begin by reviewing a method for computing
% cohomology of Cartier divisors on toric varieties which turn out to be
% extremely useful in this context.

If $X$ is a toric variety and $D$ is a torus-invariant Weil divisor on $X$
then the action of the torus $T$ on $X$ defines a grading of the cohomology
groups of $D$. At the level of global sections this is well-known allowing us
to identify the global sections of $\sO_{X}(D)$ with the lattice points of the
polytope $P_D\subseteq M_{\RR}$ corresponding to $D$. What is less well-known
is that a similar "visualization" of higher cohomology groups is also possible
thanks to a theorem of Altmann, Buczynski, Kastner and
Winz~\cite{ABKW}*{Theorem III.6} whose proof was greatly simplified by Altmann
and Ploog in~\cite{AP}*{Main Theorem} which gives us a topological
interpretation of the graded components $H^i(X,D)(m)$ for $m\in M$ and any
torus invariant Cartier divisor $D$.

To precisely describe their result we need to introduce more specific
notation. Assume $M$ is the lattice of characters of a torus $T$ and let $N$
be the lattice of one-parameter subgroups of $T$. Let
$\langle \cdot, \cdot \rangle \colon M \times N \to \ZZ$ denote the natural
pairing between them. Recall that a toric variety is specified by a rational
polyhedral fan $F$ in $N$. Let $u_1, u_2, \dotsc, u_n \in N$ be the first
lattice points in each ray of $F$ and recall~\cite{Ful} that there is a
correspondence between such rays and the torus invariant divisors $D_i$ on
$X$. If $D = \sum a_{j} D_j$ define the polytope
$P_D \subseteq M \otimes_{\ZZ} \RR$ corresponding to $D$ as
\[
  P_D \coloneq \{ m \in M \otimes \RR \mathrel{|} \langle m, u_i \rangle
  \geqslant -a_i \text{for $i = 1, 2, \dotsc, n$} \}.
\]
Recall that the global sections of $\sO_{X}(D)$ are in correspondence with the
lattice points $P_D\cap M$. There are easy combinatorial criteria for
determining when $D$ is a Cartier divisor~\cite{CLS}*{Theorem~4.2.8} and when
$D$ is nef~\cite{CLS}*{Theorem~6.3.12 and Proposition~6.1.1}. Furthermore
every torus-invariant Cartier divisor $D$ is linearly equivalent to a
difference of $T$-invariant nef divisors $D=D^{+}-D^{-}$ and any such
decomposition can be used to compute the cohomology groups of $D$ via the
following.

\begin{theorem}[{\cite{AP}*{Main Theorem}}]
  \label{t:altCoh}
  If $X$ is a projective toric variety and $D=D^{+}-D^{-}$ is a difference of
  nef divisors with corresponding polytopes
  $\Delta^{+}, \Delta^{-}\subseteq M_\RR$ then there is an isomorphism of
  vector spaces
  \[
    H^i \bigl( X,\sO_X[D] \bigr)(m) 
    \cong \widetilde{H}^{i-1} \bigl( \Delta^{-} \setminus (\Delta^{+}-m) \bigr)
  \]
  where $\widetilde{H^i}$ denotes the reduced singular cohomology groups of a
  topological space and $\Delta^{+}-m$ means the translate by $-m$ of
  $\Delta^+$ in $M_{\RR}$.
\end{theorem}

% \begin{remark}
%    For any topological space $A$ we have
%    \[
%        \widetilde{H}^{-1}(A, \RR) =
%        \begin{cases}
%            \RR & \text{if $A = \varnothing$} \\[-4pt]
%            0 & \text{otherwise.}
%        \end{cases}
%    \]
%    and that this is the only nonzero singular cohomology group when $A = \varnothing$.
% \end{remark}

\begin{proof}[Proof of Theorem~\ref{t:>tor}]
  Let $X$ be the normal toric variety defined by the normal fan of the lattice
  polygon $P+Q$. Via the usual correspondence~\cite{CLS}*{Proposition~6.1.10},
  the polygons $P$ and $Q$ define torus-invariant Weil divisors $D_P$ and
  $D_Q$ on $X$ whose corresponding sheaves have spaces of global sections
  spanned by the lattice points of $P$ and $Q$ respectively. Restricting
  sections to the points of the torus $T \subseteq X$ we obtain a
  correspondence between Laurent polynomials supported in $P$ (resp. $Q$) and
  global sections of $H^0(X,D_P)$ (resp.{} $H^0(X,D_Q)$).

  Furthermore the divisors $D_P$ and $D_Q$ are Cartier divisors with Cartier
  data given by the vertices of $P$ and $Q$~\cite{CLS}*{Theorem~4.2.8}. Since
  the vertices of $P$ and $Q$ are global sections of the corresponding line
  bundles we conclude that $D_P$ and $D_Q$ are basepoint-free and therefore
  nef divisors on $X$~\cite{CLS}*{Theorem~6.3.12 and Proposition~6.1.1}.
  Applying Theorem~\ref{t:altCoh} we obtain:
  \begin{compactenum}[1.]
  \item $h^i(X,mD) = 0$ for any positive integer $i$, any positive integer
    $m$, and any divisor $D$ in $\{D_Q,D_P,D_P+D_Q\}$. This follows from
    writing $D=D-0$.
  \item The quantity $h^1(X,D_Q-D_P)$ equals the total sum of the dimensions
    of the reduced singular cohomology groups
    $\widetilde{H^{0}}\left(P\setminus \left(Q-m\right)\right)$ as $m$ ranges
    over $\ZZ^2$. So $h^1(X,D_Q-D_P)$ agrees with the quantity $h$ defined in
    the statement of the Theorem.
  \end{compactenum}
  By Corollary~\ref{c:simp} and Theorem~\ref{t:ineq}, we conclude that a
  sufficient condition for $D_Q$ to support multipliers for $D_P$ is given by
  the inequality
  \[
    \phi_Q(2)+h > \phi_{P+Q}(-1)
  \]
  where $\phi_E(m)=\chi(\sO_X[mD_E])$ coincides with the Ehrhart polynomial of
  $E\in \{Q,P+Q\}$. By Ehrhart reciprocity (or toric Serre duality) the
  right-hand side equals the number of interior lattice points of $P+Q$,
  proving the claim.
\end{proof}

Any multiplier $g$ constructed via Theorem~\ref{t:>tor} is necessarily
nonnegative. Iterated application of Theorem~\ref{t:>tor} can therefore lead
to rational sum-of-square certificates of nonnegativity provided the
multiplier polygons $Q$ are chosen judiciously. This follows
from~\cite{BSV}*{Theorem~1.1} applied to the case of surfaces: on totally-real
non-degenerate surfaces of minimal degree, every nonnegative quadric is a sum
of squares. Surfaces of minimal degree are classified (by Bertini) and happen
to be toric, corresponding to the $2\Delta = \conv \{(0,0),(2,0),(0,2)\}$ and
Lawrence prisms of dimension $2$. A lattice polygon $S$ is a Lawrence prism
with heights $h_1,h_2$ if it is lattice congruent to
$\conv(0, e_1, h_1e_2, e_1+h_2e_2)$ for some nonnegative integers $h_1,h_2$.

\begin{example}
  \label{e:sq1}
  Let $P$ be a square in $\ZZ^2$ with side length $2$. Hence, the forms with
  support in $2P$ correspond to bihomogeneous forms in two sets of variables
  $(x_1,y_1)$ and $(x_2,y_2)$ which have degree $4$ with respect to each pair
  $(x_i,y_i)$. For the polytope $Q$, we choose a square in $\ZZ^2$ with side
  length $1$. As illustrated in Figure~\ref{fig:squares}, we have $h = 0$,
  $\# (2Q) = 9$ and $\# (P+Q)^\circ = 4$. Thus, Theorem~\ref{t:>tor}
  establishes that $Q$ supports multipliers for $P$. Since $Q$ is a Lawrence
  prism every nonnegative multiplier $g$ is a sum of squares.
\end{example}

\addtocounter{lemma}{1}
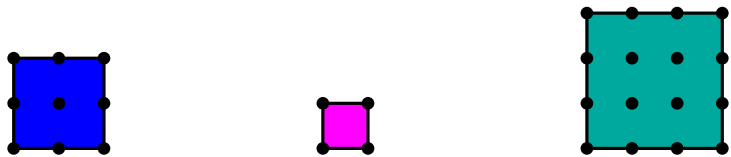
\begin{figure}[!ht]
  \centering
  \begin{tikzpicture}[scale = 0.6]
    \draw[fill=blue, very thick] (0,0) -- (2,0) -- (2,2) -- (0,2) -- (0,0) -- cycle;
    \foreach  \x in {0,1,2}
    \foreach \y in {0,1,2} 
    \fill [black] (\x,\y) circle [radius=4pt];
  \end{tikzpicture}
  \hspace{6em}
  \begin{tikzpicture}[scale = 0.6]
    \draw[fill=Magenta, very thick] (0,0) -- (1,0) -- (1,1) -- (0,1) -- (0,0) --    cycle;
    \foreach \x in {0,1}
    \foreach \y in {-0,1} 
    \fill [black] (\x,\y) circle [radius=4pt];
  \end{tikzpicture}
  \hspace{6em}    
  \begin{tikzpicture}[scale = 0.6]
    \draw[fill=Emerald, very thick] (0,0) -- (3,0) -- (3,3) -- (0,3) -- (0,0) --    cycle;
    \foreach  \x in {0,1,2,3}
    \foreach \y in {0,1,2,3} 
    \fill [black] (\x,\y) circle [radius=4pt];
  \end{tikzpicture}
  \caption{The lattice polygons $P=2Q$, $Q$, and $P+Q$}
  \label{fig:squares}
\end{figure}

Let $X \subset \PP^3$ be the toric surface corresponding to embedding
$\PP^1 \times \PP^1$ via the line bundle $\sO(1) \times \sO(1)$, and let $R$
denote the graded coordinate ring of $X$. Example~\ref{e:sq1} shows that
degree $4$ nonnegative forms on $X$ have a sum of squares multiplier of degree
$2$. We now generalize this to multiplier degree bounds for nonnegative forms
of degree $2d$ on $X$. This is done by iteratively transferring nonnegativity
to simpler polygons, until we can descend to a variety of minimal degree. This
is a warm-up to our improvement of Hilbert's bounds for ternary forms, but
polygonal geometry of rectangles is slightly simpler than that of triangles.

\begin{example}
  \label{e:biforms}
  Let $P = [0,d]^2$ be the square in $\ZZ^2$ with side length $d$, which
  corresponds to degree $d$ forms on $X = \PP^1\times \PP^1$. We let $Q$ be
  the square with side length $d-1$. Then we have $h=0$, $\#(2Q)=(2d-1)^2$ and
  $\#(P+Q)^\circ=(2d-2)^2$. Therefore, we transfer nonnegativity from $2P$ to
  $2Q$: given a nonnegative form $g_0$ with support in $2P$ we can find a
  nonnegative form $g_1$ with support in $2Q$ such that $g_0 g_1$ is a sum of
  squares. We can now apply this result to $g_1$ and continue. We produce a
  sequence of nonnegative multipliers $g_1,\dots,g_{d-1}$ with
  $g_i\in R_{2d-2i}$ such that $g_ig_{i+1}$ is a sum of squares for
  $i=0,\dots,d-2$. Note that $g_{d-1}$ is a sum of squares in $R_2$.
  It follows that for any nonnegative form $f$ of degree $2d$ on $X$ there
  exists a sum of squares multiplier $g$ with $\deg g = d(d-1)$ such that $fg$
  is a sum of squares, see Corollary~\ref{c:recursion}.

  We can improve on this bound by utilizing rectangles instead of squares.
  Concretely, we can now lower the degree from $2d$ to $2d-6$ in two steps,
  while in the previous strategy we went from $2d$ to $2d-4$.
  So let $P$ be again the square in $\ZZ^2$ with side length $d$, and take
  $Q_1 = [0,d-1]\times [0,d-2]$ to be the $(d-1)\times(d-2)$ rectangle. Then
  $h=0$, $\# (2Q_1)=(2d-1)(2d-3)$ and $\#(P+Q_1)^\circ=(2d-2)(2d-3)$.
  Therefore we can transfer nonnegativity from the $d\times d$ square to the
  $(d-1)\times (d-2)$ rectangle. Next, let $Q_2$ be the $(d-3)\times (d-3)$
  square. Then we have $h=0$, $\# (2Q_2)= (d-5)^2$ and
  $\# (Q_1+Q_2)^\circ =(d-5)\times (d-6)$. Therefore we can transfer
  nonnegativity from the $(d-1)\times (d-2)$ rectangle to the
  $(d-3)\times (d-3)$ square. So in two steps, we went from $P = [0,d]^2$ to
  $[0,d-3]^2$ improving the degree bounds faster than in the first strategy
  which only used squares.
    
  This allows us to improve degree bounds. For instance when
  $d\equiv 1 \mod 3$, we have that a nonnegative form $f$ of degree $2d$ on
  $\PP^1\times \PP^1$ has a sum of squares multiplier $g$ with
  $\deg g=d(d-1)-\frac{1}{3}d(d-1)=\frac{2}{3}d(d-1)$ such that $fg$ is a sum
  of squares, see Corollary~\ref{c:recursion}.

  We do not make any claims on optimality of these bounds, especially for high
  degree $d$. It is possible to use polygons that are different from
  rectangles, and they may lead to tighter bounds.
\end{example}

We now use the freedom of choosing multipliers with special support to give an
improvement to Hilbert's rational sum-of-squares certificates for ternary
forms. We first explain Hilbert's method. Let $\Delta$ be the right triangle
with vertices $(0,0)$, $(1,0)$ and $(0,1)$. We call the polygon
$d \cdot \Delta$ the Veronese triangle of degree $d$.

\begin{example}[Hilbert's bound for ternary forms]
  \label{e:hb}
  \cite{Hilbert93} shows that for any nonnegative ternary form $f$ of degree
  $2d$, there exists a nonnegative form $g_1$ of degree $2d-4$ such that
  $g_1f$ is a sum of squares. We can derive this result from
  Theorem~\ref{t:>tor} by setting $P = d\cdot \Delta$ to be the Veronese
  triangle of degree $d$ and $Q = (d-2)\cdot \Delta$ to be the Veronese
  triangle of degree $d-2$. We have $h=0$, $\# (2Q) = \binom{2d-2}{2}$ and
  $\# (P+Q)^\circ = \binom{2d-3}{2}$, and therefore we can transfer
  nonnegativity from $P$ to $Q$.

  We can now apply the result to $g_1$ and produce a multiplier $g_2$ of
  degree $2d-8$ such that $g_2g_1$ is a sum of squares. Applying this result
  iteratively we eventually produce a nonnegative multiplier $g_k$ of degree
  either $2$ or $4$, such that $g_kg_{k-1}$ is a sum of squares. We observe
  that $g_k$ must be a sum of squares, since nonnegative ternary quartics and
  quadrics are sums of squares by Hilbert's earlier result \cite{Hilbert88}.
  Therefore we see that a nonnegative form of degree $2d$ has a sum of squares
  multiplier $g$ such that $fg$ is a sum of squares and
  $\deg g =\frac{1}{2}d(d-2) $ when $d$ is even, and
  $\deg g=\frac{1}{2}(d-1)^2$ when $d$ is odd, see
  Corollary~\ref{c:recursion}.
\end{example}

\begin{remark}
  We cannot drop degree by more than $4$ in Hilbert's method using our
  inequality: if $P$ is the Veronese triangle of degree $d$ and $Q$ is the
  Veronese triangle of degree $d-3$, then the numbers come out to be $h = 0$,
  $\# (2Q) = \binom{2d-4}{2}$ and $\# (P+Q)^\circ = \binom{2d-4}{2}$. Thus we
  obtain equality, instead of strict inequality in Theorem~\ref{t:>tor}. For
  the cases of ternary sextics and octics, i.e. $d=3,4$ we know that we cannot
  transfer nonnegativity of degree $2d$ to degree $2d-6$ \cite{BSV19}, and
  therefore we see that the bound of Theorem~\ref{t:>tor} is tight in these
  cases.
\end{remark}

Even though we cannot immediately go from degree $d$ Veronese triangle to
degree $d-3$ Veronese triangle, we note that the inequality from degree $d-2$
Veronese triangle has some slack in it (see Example~\ref{e:hb}). Therefore, we
may choose a smaller polytope for $Q$ than the $d-2$ Veronese triangle, and
this allows us to arrive at a variety of minimal degree faster, similar to
Example~\ref{e:biforms}.

\addtocounter{lemma}{1}
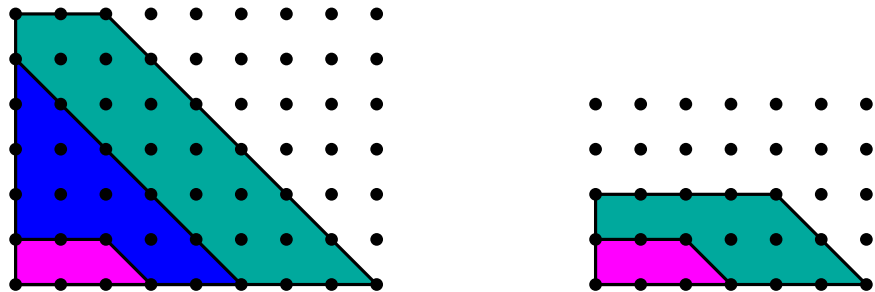
\begin{figure}[!ht]
  \centering
  \begin{tikzpicture}[scale = 0.6]
    \draw[fill=Emerald, very thick] (0,0) -- (8,0) -- (2,6) -- (0,6) -- (0,0)
    -- cycle;
    \draw[fill=blue,opacity=1.0, very thick] (0,0) -- (5,0) -- (0,5) -- (0,0)
    -- cycle;
    \draw[fill=Magenta,opacity=1.0, very thick] (0,0) -- (3,0) -- (2,1) --
    (0,1) -- (0,0) --    cycle;
    \foreach  \x in {0,1,2,3,4,5,6,7,8}
    \foreach \y in {0,1,2,3,4,5,6} 
    \fill [black] (\x,\y) circle [radius=4pt];
  \end{tikzpicture}
  \hspace{6em}
  \begin{tikzpicture}[scale = 0.6]
    \draw[fill=Emerald,opacity=1.0, very thick] (0,0) -- (6,0) -- (4,2) --
    (0,2) -- (0,0) --    cycle;
    \draw[fill=Magenta,opacity=1.0, very thick] (0,0) -- (3,0) -- (2,1) --
    (0,1) -- (0,0) --    cycle;
    \foreach  \x in {0,1,2,3,4,5,6}
    \foreach \y in {0,1,2,3,4} 
    \fill [black] (\x,\y) circle [radius=4pt];
  \end{tikzpicture}
  \caption{Left: $P$ (blue),  $Q$ (magenta) and $P+Q$ (emerald). Right: $Q$
    (magenta) and $2Q$ (emerald).}
  \label{f:hilb}
\end{figure}

\begin{example}[Improving Hilbert's bound for degree 10]
  \label{e:hb1}
  Let $P$ be the Veronese triangle of degree $5$, and $Q$ be the Lawrence
  prism with heights 3 and 2. Then, as can be seen in Figure~\ref{f:hilb},
  $\#(2Q)=18$ and $\#(P+Q)^{\circ}=20$ and furthermore $h=3$ since
  there are exactly three lattice translates of $Q$ which disconnect $P$.
  Since all nonnegative Laurent polynomials with support in $2Q$ are sums of
  squares by \cite{BSV19}, it follows that any nonnegative ternary form $f$ of
  degree 10 has a sum of squares multiplier $g$ of degree $6$ such that $gf$
  is a sum of squares. This improves Hilbert's bound from 1893, which was the
  best known bound.

  % \rs{Should we shorten the following paragraph? We have explained the
  % iteration before. }
  % Hilbert's method shows that any nonnegative ternary form of degree 10 has
  % a nonnegative multiplier $g$ of degree $6$ such that $gf$ is a sum of
  % squares. Additionally for $g$ we can find a quadratic multiplier $h$ such
  % that $hg$ is a sum of squares, and since $h$ is a quadratic form, it is a
  % sum of squares. Taken together this gives us a sum of squares multiplier
  % $hg$ such that $(hg)\cdot f$ is a sum of squares, but the degree of $hg$
  % is $8$. Here, we were able to show that multiplier $g$ can be taken to be
  % a ternary sextic of a special form, whose monomial support lies in twice a
  % Lawrence prism. Then $g$ is already a sum of squares, and we do not need
  % the additional step of finding the second multiplier $h$.

  Applying Hilbert's 1893 result iteratively to a ternary form $f$ of degree
  $10$ leads to a sum-of-squares multiplier $g$ of degree $8$. Using the
  flexibility in choosing polygons that are not Veronese triangles, our method
  shows that the multiplier $g$ can be taken to be a ternary sextic, whose
  monomial support lies in twice a Lawrence prism so that $g$ is already a sum
  of squares. Using our method, we do not need a second iteration step.
\end{example}

\begin{example}[Improving Hilbert's bound for general degrees]
  \label{e:hb2} 
  Given two nonnegative integers $d$ and $m$, let $T(d,m)$ be the lattice
  trapezoid defined by inequalities $x \geqslant 0$,
  $0 \leqslant y \leqslant d-m$, and $x+y \leqslant d$. The trapezoid $T(d,m)$
  corresponds to forms of degree $d$ vanishing to order $m$ at a
  torus-invariant point of $\PP^2$. We can think of $T(d,m)$ as the Veronese
  triangle of degree $d$ with a cut off corner.

  Let $P = T(d_1,m_1)$ and $Q = T(d_2,m_2)$ (see Figure~\ref{f:hilb}, which
  shows $T(8,2)$ in emerald on the left). Then
  \begin{align*}
    \#(2Q) &= \binom{2d_2+2}{2}-\binom{2m_2+1}{2} 
    &&\text{and}
    &\#(P+Q)^\circ &= \binom{d_1+d_2-1}{2}-\binom{m_1+m_2}{2} \, .
  \end{align*}
  We take $d_2 \leqslant d_1$. 
  
  % If $d_2-m_2<d_1-m_1 $then the ``correction term" $h$ is equal to zero
  % \gb{this seems right}. If $d_2-m_2 \geqslant d_1-m_1$ then
  % $h=\binom{m_2-m_1+1}{2}$ \gb{may need to be more careful about $h$}.

  As we observed in Example \ref{e:hb}, Hilbert's proof took $m_1=m_2=0$ and
  $d_2=d_1-2$, and choosing $m_1=m_2=0$ and $d_2=d_1-3$ is not possible, since
  the strict inequality required to apply Theorem \ref{t:>tor} is an equality.
  However, we can decrease the degree by $3$, if we first ``bite off" a corner
  of the Veronese triangle.
    
  This leads to the following procedure: take one step to ``bite" off a corner
  of the degree $d$ Veronese triangle $d\Delta$ as much as possible. At the
  start we have $m_1=0$ and $d_1=d_2=d$, so that
  $\#(P+Q)^\circ=\binom{2d-1}{2}-\binom{m_2}{2}$ and
  $\#(2Q)=\binom{2d+2}{2}-\binom{2m_2+1}{2}$. We can make
  $m_2 \approx 2\sqrt{d}$. For the next step we have $d_1=d$ and
  $m_1 \approx 2\sqrt{d}$. Then we can take $d_2=d_1-3$ and $m_2=m_1-3$:
  \[
    \#(2Q)-\#(P+Q)^\circ = \binom{2m_1-4}{2}-\binom{2m_1-5}{2} > 0 \, .
  \]
  From this inequality we see that we can continue decreasing the degree by
  $3$ until the ``bitten off" corner (i.e. $m_1$) falls below $3$. Then we
  repeat.
    
  In this process we take roughly $d/3$ steps and, in each step (except
  roughly $\sqrt{d}$ ``biting off" steps), we decrease the degree by 3.
  Therefore, the total degree of the multiplier is bounded by
  $\frac{d^2}{6} + \textit{lower order terms}$, which is an asymptotic
  improvement over Hilbert's bound, which has a leading order of
  $\frac{d^2}{4}$.
\end{example}

%%%%%%%%%%%%%%%%%%%%%%%%%%%%%%%%%%%%%%%%%%%%%%%%%%%%%%%%%%%%%%%%%%%%%%%%%%%%%%
\section{Applications to del Pezzo Surfaces}
\label{s:pezzos}

\noindent
We now characterize nonnegative global sections for all even divisors on
totally-real del Pezzo surfaces having degree at least $3$. Beyond their
prominence in the theory of algebraic surfaces, del Pezzo surfaces are
interesting within real algebraic geometry because they admit several distinct
real structures. More significantly, these surfaces are a successor to
varieties of minimal degree (surfaces $X$ in $\PP^n$ such that
$\deg(X) = 1 + \codim(X)$); see \cite{Dolgachev}*{\S8.1}. Indeed, the linearly
normal nonsingular surfaces of almost minimal degree (surfaces $X$ in $\PP^n$
such that $\deg(X) = 2 + \codim(X)$) are del Pezzo surfaces of degree at least
$3$ embedded via their anticanonical linear series; see
\cite{Dolgachev}*{\S8.3}. From this perspective, our characterization extends
the degree bounds for sum-of-squares multipliers on surfaces of minimal degree
in \cite{BSV19}*{Theorem 1.2}.

The theorem in this section encapsulates the major insights by showing that
certificates of nonnegativity for global sections of a divisor $2D$ may always
be obtained from analogous certificates for simpler divisors.

\begin{theorem}
  \label{t:delPezzo} 
  Let $X$ be a totally-real del Pezzo surface having degree at least $3$ and
  canonical divisor $K_X$. For any nonzero real effective divisor $D$ on $X$,
  there exists a finite sequence $D_0, D_1, \dotsc, D_k$ of effective divisors
  on $X$ with $D_0 = D$ such that $-K_{X} \cdot D_{i} < -K_{X} \cdot D_{i-1}$,
  $D_{i}$ supports multipliers for $D_{i-1}$ for all
  $1 \leqslant i \leqslant k$, and $D_k$ is either zero or a positive multiple
  of a conic bundle. In particular, the length $k$ of the sequence is bounded
  above by $-K_{X} \cdot D$.
\end{theorem}

Before delving into the proof, we recount some features of del Pezzo surfaces.
A \emph{del Pezzo surface} is a nonsingular geometrically-irreducible surface
$X$ whose anticanonical divisor $-K_{X}$ is ample; see
\cite{Dolgachev}*{Definition~8.1.2}. Its \emph{degree} is the
self-intersection number $d \coloneq K_{X} \cdot K_{X}$, which satisfies
$1 \leqslant d \leqslant 9$; see \cite{Dolgachev}*{Proposition~8.1.6}. Over
$\CC$, del Pezzo surfaces form a single sequence with one exception. Other
than $\PP^1 \mathbin{\times} \PP^1$ which has degree $8$, a del Pezzo surface
is a blow-up of $\PP^2$ at $9-d$ general points: no three lie on a line, no
six line on a conic, and no eight lie on a singular cubic with one of the
points at the singularity; see \cite{Dolgachev}*{Proposition~8.1.18}. Hence,
the del Pezzo surfaces having degree $d$ greater than $5$ are all normal toric
varieties. However, there are infinitely many nonisomorphic del Pezzo surfaces
of each degree less than $5$; see \cite{Dolgachev}*{Sections~8.5--8.8}.

Over the complex numbers, the birational geometry of del Pezzo surfaces is
comparatively simple. The Picard group of a del Pezzo surface is a free
abelian group of rank $10-d$. For the special case
$\PP^1 \mathbin{\times} \PP^1$, the Picard group is generated by the divisor
classes of the two rulings. Otherwise,
$X = \Bl_{p_1, p_2, \dotsc, p_{9-d}}(\PP^2)$ and the Picard group $\Pic(X)$ is
generated by the pullback $H$ of the hyperplane class on $\PP^2$ (under the
canonical morphism $\pi \colon X \to \PP^2$) and the classes of the
exceptional divisors $E_1, E_2, \dotsc, E_{9-d}$, which are the preimages of
the points $p_1, p_2, \dotsc, p_{9-d}$. In particular, we have
$K_{X} = -3H +E_{1} + E_{2} + \dotsb + E_{9-d}$. Moreover, the intersection
product is determined by $H \cdot H = 1$, $H \cdot E_{i} = 0$ for all
$1 \leqslant i \leqslant 9-r$, and $E_{i} \cdot E_{\!j} = - \delta_{i,j}$ for
all $1 \leqslant i \leqslant j \leqslant 9-r$. A \emph{$(-1)$-curve} on a
surface $X$ is a divisor class $C$ satisfying $C \cdot C = -1$ and
$K_{X} \cdot C = -1$ where $K_{X}$ is the canonical divisor on $X$. When
$1 \leqslant d \leqslant 7$, the cone of curves on a del Pezzo surface of
degree $d$ is the closed cone in the real vector space
$\Pic(X) \otimes_{\ZZ} \RR$ generated by the classes of $(-1)$-curves; see
\cite{Dolgachev}*{Theorem~8.2.19}. Moreover, the finitely many $(-1)$-curves
are explicitly enumerated in \cite{Dolgachev}*{Proposition~8.2.15}.

Over the real numbers, the classification of del Pezzo surfaces is more
involved because a complex del Pezzo surface may have more than one real
structure. The two del Pezzo surfaces of degree $8$ in the subsequent example
begin to reveal some of the intricacies; see \cite{Rus}*{Proposition~1.2}.

\begin{example}
  \label{e:quad}
  Consider the totally-real surfaces $Q^{2,2}$ and $Q^{3,1}$ in $\PP^3$
  defined by $x_0^2+x_1^2-x_2^2-x_3^2$ and $x_0^2+x_1^2+x_2^3-x_3^2$
  respectively. Over $\CC$, these subvarieties are isomorphic (in fact,
  projectively equivalent) because their defining quadratic polynomials have
  the same rank. No such isomorphism exists over $\RR$, because the quadratic
  polynomials have different signatures. Topologically, the set $Q^{2,2}(\RR)$
  of real points is the torus $S^1 \mathbin{\times} S^1$ and $Q^{3,1}(\RR)$ is
  the sphere $S^2$ which are not homeomorphic. Geometrically, the real variety
  $Q^{2,2}$ is ruled by real lines, whereas conjugation on $\PP^3$ exchanges
  the complex lines of the two rulings through each real point of $Q^{3,1}$.
\end{example}

More generally, a real structure on a complex variety is a choice of an
antiholomorphic involution. The real points are, by definition, the subset of
points fixed by the involution. For any real scheme $X$, its complexification
$X_{\CC} \coloneq X \times_{\Spec(\RR)} \Spec(\CC)$ has a canonical
antiholomorphic involution induced by complex conjugation on $\CC$. The study
of equivalence classes of real structures on a complex projective variety
$X_{\CC}$ is equivalent to the study of isomorphism classes of real projective
varieties $X$ whose complexification is isomorphic over $\CC$ to $X_{\CC}$;
see \cite{Rus}*{\S1}. The ensuing example, constructed via an antiholomorphic
involution, manifests a real del Pezzo surface of degree $4$ whose real points
form a disconnected topological space; see \cite{Rus}*{Example~2} for further
details.

\begin{example}
  \label{e:DeJon} 
  Let $\Gamma$ be a nonsingular real plane cubic curve having two real
  components. Choose a general real point $p_1$ on $\Gamma$. The intersection
  of $\Gamma$ with its polar curve with respect to $p_1$ has degree $6$. Since
  $p_1$ is a general point on $\Gamma$, this intersection contains of four
  distinct real points $p_2, p_3, p_4, p_5$ on $\Gamma$ (in addition to $p_1$)
  such that the tangent to $\Gamma$ at $p_i$ passes through $p_1$ for all
  $2 \leqslant i \leqslant 5$; see \cite{Dolgachev}*{Theorem~1.1.1}. Let
  $\mathbb{D}_{\CC} \coloneq \Bl_{p_1, p_2, \dotsc, p_5}(\PP^2)$ be the
  associated complex del Pezzo surface of degree $4$. The de Jonqui\`eres
  birational involution of $\PP^2$ is uniquely determined by the property that
  its restriction to a general line $L$ passing through $p_1$ coincides with
  the involution of $\PP^1$ that interchanges the residual intersection points
  of $L$ with $\Gamma$ and fixes $p_1$; see \cite{Dolgachev}*{\S7.3.6}. This
  birational involution lifts to an antiholomorphic involution
  $\tau \colon \mathbb{D}_{\CC} \to \mathbb{D}_{\CC}$ that sends $E_1$ to
  $2H - E_1 - E_2 - E_3 - E_4 - E_5$ and sends $E_i$ to $H - E_1 - E_i$ for
  all $2 \leqslant i \leqslant 5$. Hence, every $(-1)$-curve $C$ on
  $\mathbb{D}$ satisfies $C \cdot \tau(C) = 1$. Setting $\mathbb{D}$ to be the
  totally-real del Pezzo surface corresponding to this antiholomorphic
  involution on $\mathbb{D}_{\CC}$, we see that its set of real points is
  $S^{2} \sqcup S^{2}$.
\end{example}

To catalogue the relevant real structures on complex del Pezzo surfaces, we
collect some notation. For nonnegative integers $a$ and $b$, let $Q^{a,b}$ be
the real subvariety in $\PP^{a+b-1}$ defined by the quadratic polynomial
$x_0^2 + x_1^2 + \dotsb + x_{a-1}^2 - x_{a}^2 - x_{a+1}^2 - \dotsb -
x_{a+b-1}^2$ in $\RR[x_0,x_1, \dotsc, x_{a+b-1}]$. For any real surface $X$,
let $X(a,2b)$ be the real surface obtained from $X$ by blowing-up $a$ distinct
real points and $b$ pairs of conjugate nonreal points. With these definitions,
we have $Q^{2,2} \cong Q^{2,1} \mathbin{\times} Q^{2,1}$,
$\PP^2(2,0) \cong Q^{2,2}(1,0)$, and $\PP^2(0,2) \cong Q^{3,1}(1,0)$.
Table~\ref{f:realDel} lists the $24$ totally-real del Pezzo surfaces of degree
at least $3$; see \cite{Rus}*{Proposition~1.2 and Corollaries~2.4,~3.2--3.3}
or \cite{Kollar}*{Proposition~86} for a complete classification including
those containing no real points. In Table~\ref{f:realDel}, the column heading
"$\rho(X_{\RR})$" stands for the rank of the real Picard group of $X$ and the
column heading "$\#$ real $(-1)$'s" is an abbreviation for the number of real
$(-1)$-curves on $X$.

\addtocounter{lemma}{1}
\begin{table}[!ht]
  \centering
  \caption{Totally-real del Pezzo surfaces of degree at least $3$}
  \label{f:realDel}
  \vspace{-1em}
  \begin{tabular}[t]{cccc} \hline
    Degree & $X$ & $\rho(X_{\RR})$ & $\#$ real $(-1)$'s \\ \hline \\[-10pt]
    $9$ & $\PP^2$ & $1$ & $0$ \\[7pt]
    $8$ & $\PP^2(1,0)$ & $2$ & $1$ \\[2pt]
    $8$ & $Q^{2,2}$ & $2$ & $0$ \\[2pt]
    $8$ & $Q^{3,1}$ & $1$ & $0$ \\[7pt]
    $7$ & $\PP^2(2,0)$ & $3$ & $3$ \\[2pt]
    $7$ & $\PP^2(0,2)$ & $2$ & $1$ \\[7pt]
    $6$ & $\PP^2(3,0)$ & $4$ & $6$ \\[2pt]
    $6$ & $\PP^2(1,2)$ & $3$ & $6$ \\[2pt]
    $6$ & $Q^{3,1}(0,2)$ & $2$ & $0$ \\[2pt]
    $6$ & $Q^{2,2}(0,2)$ & $3$ & $0$ \\[7pt]
    $5$ & $\PP^2(4,0)$ & $5$ & $10$ \\[2pt]
    $5$ & $\PP^2(2,2)$ & $4$ & $4$ \\[2pt]
    \hline
  \end{tabular}
  \hspace{30pt}
  \begin{tabular}[t]{cccc} \hline
    Degree & $X$ & $\rho(X_{\RR})$ & $\#$ real $(-1)$'s  \\ \hline \\[-10pt]    
    $5$ & $\PP^2(0,4)$ & $3$ & $2$ \\[12pt] 
    $4$ & $\PP^2(5,0)$ & $6$ & $16$ \\[2pt]
    $4$ & $\PP^2(3,2)$ & $5$ & $8$ \\[2pt]
    $4$ & $\PP^2(1,4)$ & $4$ & $4$ \\[2pt]
    $4$ & $Q^{3,1}(0,4)$ & $3$ & $0$ \\[2pt]
    $4$ & $Q^{2,2}(0,4)$ & $4$ & $0$ \\[2pt]
    $4$ & $\mathbb{D}$ & $2$ & $0$ \\[12pt] 
    $3$ & $\PP^2(6,0)$ & $7$ & $27$ \\[2pt]
    $3$ & $\PP^2(4,2)$ & $6$ & $15$ \\[2pt]
    $3$ & $\PP^2(2,4)$ & $5$ & $7$ \\[2pt]
    $3$ & $\PP^2(0,6)$ & $4$ & $3$ \\[2pt]
    $3$ & $\mathbb{D}(1,0)$ & $3$ & $3$ \\[2pt]
    \hline
  \end{tabular}
\end{table}

Over the real numbers, the birational geometry of surfaces is more
complicated: the real Picard group may have smaller rank and there are more
possibilities for the extremal rays in the cone of curves. Conic bundles
provide one new kind of extremal ray. On a del Pezzo surface $X$, a
\emph{conic bundle} is a divisor class $B$ such that $B \cdot B = 0$ and
$-K_{X} \cdot B = 2$. By the Riemann--Roch Theorem, the complete linear series
of $B$ defines a surjective morphism $\pi_{B} \colon X \to \PP^1$ such that
every fibre is isomorphic to a plane conic. As \cite{Kollar}*{Theorem~29}
establishes that a conic bundle can be an extremal ray only when the rank of
the real Picard group is $2$, the following example shows that inventorying
the minimal conic bundles is relatively straightforward.

\begin{example}
  \label{e:min}
  From Table~\ref{f:realDel}, we see that there are $5$ totally-real del Pezzo
  surfaces with real Picard rank equal to $2$. Since any $2$-dimensional
  closed convex cone has two extremal rays, there are just two divisors
  classes on each surface to analyze.
  \begin{compactitem}[$\bullet$]
  \item Suppose that $X = \PP^2(1,0)$. Let $H$ denote the pullback of the
    hyperplane class on $\PP^2$ and let $E_1$ be the exceptional divisor over
    the distinguished real point in $\PP^2$. The extremal rays on $X$ are the
    real $(-1)$-curve $E_1$ and the real conic bundle $B \coloneq H - E_1$.
    Moreover, \cite{Kollar}*{Theorem~29} implies that
    $\pi_{B} \bigl( X(\RR) \kern-1.0pt \bigr) = \PP^1(\RR)$.
  \item Suppose that $X = Q^{2,2} \cong Q^{2,1} \mathbin{\times} Q^{2,1}$. The
    extremal rays on $X$ are given by the divisor classes of the two real
    rulings which are real conic bundles. In either case, we have
    $\pi_{B} \bigl( X(\RR) \kern-1.0pt \bigr) = \PP^1(\RR)$.
  \item Suppose that $X = \PP^2(0,2) \cong Q^{2,2}(1,0)$. One extremal ray
    contracts the real $(-1)$-curve and the other contracts the disjoint pair
    of conjugate exceptional curves, so there is no conic bundle. As overkill,
    \cite{Kollar}*{Theorem~29} also proves that a conic bundle can only be an
    extremal ray on a del Pezzo surface having even degree.
  \item Suppose that $X = Q^{3,1}(0,2)$. Let $L_1$ and $L_2$ be the pullback
    of the two rulings on $Q^{3,1}$ and let $E_1$ and $E_2$ be the exceptional
    divisors over the distinguished pair of conjugate nonreal points. One
    extremal ray contracts the disjoint pair of conjugate exceptional curves.
    The second is the real conic bundle $B \coloneq L_1 + L_2 - E_1 - E_2$.
    \cite{Kollar}*{Theorem~29} confirms that $\pi_{B} \colon X \to \PP^{1}$
    has two singular fibres, so the image
    $\pi_{B} \bigl( X(\RR) \kern-1.0pt \bigr)$ is a closed interval in
    $\PP^1(\RR)$ whose endpoints correspond to the singular fibres.
  \item Suppose that $X = \mathbb{D}$. Let $H$ denote the pullback of the
    hyperplane class on $\PP^2$ and let $E_1, E_2, \dotsc, E_5$ be the
    exceptional divisors over the special real points in $\PP^2$. The extremal
    rays on $X$ are the two real conic bundles $H - E_1$ and
    $2H - E_2 - E_3 - E_4 - E_5$. In both cases, \cite{Kollar}*{Theorem~29}
    establishes that $\pi_{B} \colon X \to \PP^{1}$ has four singular fibres,
    so the image $\pi_{B} \bigl( X(\RR) \kern-1.0pt \bigr)$ consists of $2$
    disjoint closed intervals in $\PP^1(\RR)$ whose endpoints correspond to
    the singular fibres. \qedhere
  \end{compactitem}
\end{example}

\begin{remark}
  \label{r:conic}
  Applying the minimal model program for real algebraic
  surface~\cite{Kollar}*{Theorem~30}, we see that every totally-real del Pezzo
  surface $X$ of degree at least $3$ is obtained from $\PP^2$, $Q^{2,2}$,
  $Q^{3,1}$, or $\mathbb{D}$ from a sequence of blow-ups at either a real
  point or a pair of conjugate nonreal points. The birational map associated
  to either type of blow-up is strongly dominant over $\RR$. It follows that
  any real conic bundle on $X$ is the pullback of a minimal conic bundle
  appearing in Example~\ref{e:min}.
\end{remark}

Global sections of a real conic bundle may require modified certificates of
nonnegativity.

\begin{remark}
  \label{r:modSOS} 
  Let $B$ be a real conic bundle on a totally-real del Pezzo surface $X$ and
  let $\pi_{B} \colon X \to \PP^1$ be the associated surjective morphism. For
  any positive integer $c$, consider $f$ in
  $\smash{H^0 \kern-1.0pt\bigl( X, \sO_{X}(2cB) \kern-1.0pt \bigr)}$. The
  global section $f$ is the pullback of a unique homogeneous polynomial $g$ in
  $\RR[x_0,x_1]$ of degree $2c$. Moreover, $f$ is nonnegative if and only if
  $g$ is nonnegative on $\pi_{B} \bigl( X(\RR) \kern-1.0pt \bigr)$. When
  $\pi_{B} \bigl( X(\RR) \kern-1.0pt \bigr) = \PP^1(\RR)$, every nonnegative
  $g$ can be expressed as a sum of squares in $\RR[x_0,x_1]$. For the
  remaining cases, choose real coordinates on $\PP^1$ such that
  $[1 \mathbin{:} 0]$ is not in $\pi_{B} \bigl( X(\RR) \kern-1.0pt \bigr)$.
  Under the map $[x_0 \mathbin{:} x_1] \mapsto x_0/x_1$, the image
  $\pi_{B} \bigl( X(\RR) \kern-1.0pt \bigr)$ corresponds to the closed
  interval $[a_0,a_1]$ or the disjoint union $[a_0,a_1] \sqcup [a_2,a_3]$. In
  the first case, every nonnegative $g$ can be expressed as
  $h_0 + h_1(a_0 \, x_1 - x_0)(x_0 - a_1 \, x_1)$ where $h_0$ and $h_1$ are
  sums of squares in $\RR[x_0,x_1]$. In second case, every nonnegative $g$ can
  be expressed as
  $h_0 + h_1 (x_0 - a_1 \, x_1) (x_0 - a_2 \, x_1) + h_2(a_0 \, x_1 - x_0)(x_0
  - a_4 \, x_1)$ where $h_0$, $h_1$, and $h_2$ are sums of squares.
\end{remark}

Given this background on totally-real del Pezzo surfaces having degree at
least $3$, we now present four lemmas about divisors needed for our proof of
Theorem~\ref{t:delPezzo}. A divisor $D$ on a surface $X$ is \emph{nef} if
$D \cdot C \geqslant 0$ for any effective divisor $C$.

\begin{lemma} 
  \label{l:tri} 
  Let $X$ be a totally-real del Pezzo surface of degree at least $3$. Assume
  that $D$ is a nonzero effective divisor on $X$. When the divisor $D$ is not
  nef, there exists a divisor $E$ on $X$, which is either a real $(-1)$-curve
  or the sum of a disjoint pair of conjugate complex $(-1)$-curves, such that
  $\smash{H^0 \kern-1.0pt \bigl( X, \sO_{X}(D) \kern-1.0pt \bigr)} =
  \smash{H^0 \kern-1.0pt \bigl( X, \sO_{X}(D-E) \kern-1.0pt \bigr)}$. When the
  divisor $D$ is nef but not ample, then either
  \begin{compactenum}[\upshape (i)]
  \item there exists a real $(-1)$-curve $C$ such that $D \cdot C = 0$,
  \item there is a pair $(C,\overline{C})$ of disjoint conjugate $(-1)$-curves
    such that $D \cdot C = D \cdot \overline{C} = 0$, or
  \item $D$ is a positive multiple of a conic bundle.
  \end{compactenum}
  In cases \textup{(i)} and \textup{(ii)}, there exists a real birational
  morphism $\pi \colon X \to X'$ from $X$ to a totally-real del Pezzo surface
  $X'$ of larger degree than $X$ and a nef divisor $N$ on $X'$ such that
  $D = \pi^*(N)$.
\end{lemma}

\begin{proof} 
  Since the divisor $D$ is not nef, there exists a $(-1)$-curve $C$ such that
  $D \cdot C < 0$. Assuming that $C \neq \overline{C}$ and
  $C \cdot \overline{C} \neq 0$, the divisor $B \coloneq C + \overline{C}$
  would be a real conic bundle satisfying $D \cdot B < 0$, which contradicts
  the hypothesis that $D$ is effective because every conic bundle is nef. It
  follows that either $C$ is real or $(C, \overline{C})$ is a disjoint pair of
  conjugate $(-1)$-curves. Let $E$ be the real divisor defined by either $C$
  or $C + \overline{C}$. The long exact sequence in cohomology associated to
  the short exact sequence
  \[
    \begin{tikzcd}[column sep = 2.0em]
      0 \ar[r]
      & \sO_{X}(D-E) \ar[r]
      & \sO_{X}(D) \ar[r]
      & \sO_{E}(D) \ar[r]
      & 0 \, 
    \end{tikzcd}
  \]
  yields the desired equality of global sections.

  Suppose that $D$ is nef but not ample. When $d \leqslant 7$, there exists a
  $(-1)$-curve $C$ such that $D \cdot C = 0$. Assuming that
  $C \neq \overline{C}$ and $C \cdot \overline{C} \neq 0$, the divisor
  $B \coloneq C + \overline{C}$ is a real conic bundle satisfying
  $D \cdot B = 0$ which implies that $D$ is a positive multiple of $B$. When
  either $C$ is real or $(C, \overline{C})$ is a disjoint pair of conjugate
  $(-1)$-curves, let $E$ be the real divisor defined by either $C$ or
  $C + \overline{C}$. The target of the real birational morphism
  $\pi \colon X \to X'$ that contracts $E$ is a totally-real del Pezzo surface
  of higher degree. Since $D \cdot E = 0$, we see that $D$ is a pullback of a
  nef divisor on $X'$. Lastly, when $8 \leqslant d \leqslant 9$, the
  hypothesis that $D$ is nef but not ample implies that $d = 8$ and there are
  two options: $D$ is either a positive multiple of $H - E_1$ on
  $X = \PP^1(1,0)$ or a positive multiple of a real ruling on $X = Q^{2,2}$.
  Thus, Example~\ref{e:min} implies that the $D$ is a positive multiple of a
  conic bundle in these cases.
\end{proof}

Every ample divisor $D$ on a complex del Pezzo surface having degree at least
$3$ can be written as $D = A + N$ where $A$ is the \emph{minimal ample
  divisor} defined in Table~\ref{f:minAmple} and $N$ is some nef divisor. When
$X$ is a real, the minimal ample divisor is real, so the nef divisor $N$ is
also real.

\addtocounter{lemma}{1}
\begin{table}[!ht]
  \centering
  \caption{Minimal ample divisor on complex del Pezzo surfaces}
  \label{f:minAmple}
  \vspace{-0.5em}
  \begin{tabular}{ccc} \hline
    Degree & $X$ & $A$  \\ \hline \\[-10pt]
    $9$ & $\PP^2$ & $H = -\tfrac{1}{3} K_X$ \\[5pt]
    $8$ & $\Bl_{p_1}(\PP^2)$ & $2H - E_1$ \\
    $8$ & $\PP^1 \times \PP^1$ & $L_1 + L_2 = -\tfrac{1}{2} K_X$ \\[5pt]
    $d \leqslant 7$ & $\quad \Bl_{p_1, p_2, \dotsc, p_{9-d}}(\PP^2) \quad$
                 & $-K_X$ \\[2pt] \hline
  \end{tabular}
\end{table}

\begin{lemma}
  \label{l:step} 
  Let $X$ be a totally-real del Pezzo surface of degree $d$ at least $3$ and
  let $A$ denote the minimal ample divisor on $X$. There exists a real
  effective divisor $C$ and nef divisors $N$ and $M$ such that $A = C + N$,
  $-K_X = C + M$ and a general section of $M$ is a smooth rational curve. When
  $d \leqslant 7$ or $X = \PP^2(1,0)$, the divisor $C$ can be chosen to be a
  real $(-1)$-curve or a real conic bundle.
\end{lemma}

\begin{proof} 
  Suppose that $8 \leqslant d \leqslant 9$. When $X = \PP^2$, setting
  $C \coloneq H$, $N \coloneq 0$, and $M \coloneq 2H$ implies that
  $A = C + N$, $-K_{X} = C + M$, and a general section of $M$ is smooth
  rational curve. When $X = \PP^2(1,0)$, setting $C \coloneq H - E_1$,
  $N \coloneq H$, and $M \coloneq 2H$ ensures that $A = C + N$, $-K_{X} = C+M$
  and the genus formula~\cite{Bea}*{I.15} shows that a section of $M$ has
  genus zero. When $X$ is $Q^{2,2}$ or $Q^{3,1}$, let $L_1$ and $L_2$ be the
  divisors classes of the two rulings. Setting $C \coloneq A$, $N \coloneq 0$,
  and $M \coloneq L_1 + L_2$, we see that $A = C + N$, $-K_{X} = C + M$, and
  the genus formula again shows that a general section of $M$ is a smooth
  rational curve.

  Suppose that $3 \leqslant d \leqslant 7$. When the surface $X$ contains a
  real $(-1)$-curve $C$, set $N \coloneq -K_{X} - C$ and $M \coloneq N$. For
  any $(-1)$-curve $C'$, we have $N \cdot C' = 1 - C \cdot C'$. Since
  $d \geqslant 3$, any two $(-1)$-curves on $X$ intersect in a most one point;
  see \cite{Dolgachev}*{Proposition~8.2.15}. It follows that $N$ is nef.
  Combining the definition of a $(-1)$-curve and the genus formula, we deduce
  that general section of $M$ is a smooth rational curve. When the surface $X$
  does not contain a real $(-1)$-curve, Table~\ref{f:realDel} shows that
  $d \geqslant 4$. Moreover, Example~\ref{e:min} and Remark~\ref{r:conic} also
  establish that $X$ has a real conic bundle $B$. Set $C \coloneq B$,
  $N \coloneq -K_{X} - B$, and $M \coloneq N$ so that $A = C + N$ and
  $-K_{X} = C + M$. The divisor class $N$ is nef because no del Pezzo surface
  of degree at least $4$ contains a triangle which implies that
  $B \cdot C' \leqslant 1$ for any $(-1)$-curve $C'$; see
  \cite{Dolgachev}*{\S8.4.1, \S8.4.2, \S8.5.1, \S8.6.3}. Combining the
  definition of a conic bundle and the genus formula, we see that general
  section of $M$ is a smooth rational curve.
\end{proof}

\begin{lemma}
  \label{l:nefEff} 
  Every nef divisor $N$ on a del Pezzo surface $X$ is effective. Moreover, for
  any positive integer $i$ and any nonnegative integer $m$, we have
  $h^i(X, m \, N) = 0$.
\end{lemma}

\begin{proof} 
  Since $-K_{X}$ is ample, the divisor class $N - \tfrac{1}{m} K_{X}$ is ample
  for any positive integer $m$; see \cite{PAG}*{Corollary~1.4.10}. Hence, the
  Nakai Criterion~\cite{PAG}*{Theorem~1.2.23} establishes that
  $(m N - K_{X})^2 > 0$ for any nonnegative integer $m$, which implies that
  $m N - K_{X}$ is big and nef; see \cite{PAG}*{Theorem~2.2.16}. The
  Kawamata--Viehweg Vanishing Theorem~\cite{PAG}*{Theorem~4.3.1} demonstrates
  that $h^i \bigl( X, (m N - K_{X}) + K_{X} \bigr) = 0$ for any positive
  integer $i$. The effectiveness of $m N$ then follows from the Riemann--Roch
  Theorem.
\end{proof}

\begin{lemma}
  \label{l:delAmp} 
  For any ample divisor $D$ on a totally-real del Pezzo surface of degree at
  least $3$, there exists a nonzero effective divisor $C$ such that
  $E \coloneq D-C$ is effective and supports multipliers for $D$.
\end{lemma}

\begin{proof} 
  Let $X$ be a totally-real del Pezzo surface of degree $d$ where
  $d \geqslant 3$. Lemma~\ref{l:step} proves that there exists a nonzero
  effective divisor $C$ on $X$ and nef divisors $N$ and $M$ on $X$ such that
  $A = C + N$, $-K_{X} = C + M$, and a general section of $M$ is a smooth
  rational curve. We claim that the divisor $E \coloneq D - C$ is effective
  and supports multipliers for $D$.

  Suppose that $E = 0$. Since $M \neq 0$, it follows that $d \geqslant 8$ and
  there are four cases: the pair $(X,D)$ is $(\PP^2, H)$, $(\PP^2, 2H)$,
  $\bigl( \Bl_{p_1}(\PP^2), 2H-E_1 \bigr)$ or
  $(\PP^1 \mathbin{\times} \PP^1, L_1 + L_2)$. In all of these cases, the
  nonnegative global sections of $\mathcal{O}_{\!X}(2D)$ coincide with the
  sums of squares because the surface $X$ embedded the very ample line bundle
  $\mathcal{O}_{\!X}(D)$ is a variety of minimal degree; see
  \cite{BSV}*{Theorem~1.1}.

  Suppose that $E \neq 0$. To prove the claim, it is enough to verify the
  hypotheses of Corollary~\ref{c:simp}. Since the divisor $D$ is ample, there
  exists a unique nef divisor $N'$ such that $D = A + N'$. It follows that
  $E = D - C = N + N'$ is nef and Lemma~\ref{l:nefEff} shows that $E$ is
  effective and $h^i(X, mE) = 0$ for any positive integers $i$ and $m$. The
  divisor $E-D = -C$ has negative intersection with the ample divisor $-K_{X}$
  which gives $h^0(X, E-D) = 0$. The assumption that $d \geqslant 3$ ensures
  that the minimal ample divisor $A$ on $X$ is very ample, so the divisor
  $D = A + N'$ is free and the divisor $D+E = A+N+2N'$ is very ample. Hence,
  Lemma~\ref{l:nefEff} also shows that $h^i(X, mD+mE) = 0$ for any positive
  integers $i$ and $m$. All that remains is to confirm the inequality
  $h^0(X,2E) + h^1(X,E-D) > h^0(X,K_{X}+D+E)$. The choice of $C$ ensures that
  $K_{X}+D+E = 2E - M$. Consider the short exact sequence
  \[
    \begin{tikzcd}[column sep = 2.0em]
      0 \ar[r]
      & \sO_{X}(2E-M) \ar[r]
      & \sO_{X}(2E) \ar[r]
      & \sO_{M}(2E|_{M}) \ar[r]
      & 0 \, . 
    \end{tikzcd}
  \] 
  As $D+E = A+N+2N'$ is big and nef, the Kawamata--Viehweg Vanishing Theorem
  demonstrates that $h^1(X, 2 E-M) = h^1(X, K_{X} + D + E) = 0$, so we obtain
  the exact sequence
  \[
    \begin{tikzcd}[column sep = 2.0em]
      0 \ar[r]
      & \smash{H^0 \kern-1.0pt \bigl( X, \sO_X(2E-M) \kern-1.0pt \bigr)} \ar[r]
      & \smash{H^0 \kern-1.0pt \bigl( X, \sO_X(2E) \kern-1.0pt \bigr)} \ar[r]
      & \smash{H^0 \kern-1.0pt \bigl( M, \sO_{M}(2E|_{M}) \kern-1.0pt \bigr)} \ar[r]
      & 0 \, .
    \end{tikzcd}
  \]
  The divisor $2E = 2(N + N')$ being nef establishes that
  $2E \cdot M \geqslant 0$. It follows that $2E|_{M}$ is a divisor of
  nonnegative degree on the smooth rational curve $M$, so
  $h^0\kern-1.0pt \bigl( M, \sO_{M}(2E|_{M}) \kern-1.0pt \bigr) > 0$. We
  conclude that
  $h^0(X,2E) + h^1(X,E-D) \geqslant h^0(X, 2E) > h^0(X, 2E-M) =
  h^0(X,K_{X}+D+E)$.
\end{proof}
    
We now prove the main result of this section. 

\begin{proof}[Proof of Theorem~\ref{t:delPezzo}] 
  Suppose that $D$ is not nef. Lemma~\ref{l:tri} shows that there exists a
  divisor $E$ on $X$, which is either a real $(-1)$-curve or the sum of a
  disjoint pair of conjugate $(-1)$-curves, such that
  $\smash{H^0 \kern-1.0pt \bigl( X, \sO_{X}(D) \kern-1.0pt \bigr)} =
  \smash{H^0 \kern-1.0pt \bigl( X, \sO_{X}(D-E) \kern-1.0pt \bigr)}$.
  Iterating this step a finite number of times, we reach a nef divisor $D'$.
  By construction, we have $-K_{X} \cdot D' < -K_{X} \cdot D$ and $D'$
  supports multipliers for $D$.

  For some nonnegative integer $j$, we may assume that there exists divisors
  $D_0, D_1, \dotsc, D_j$ with $D_0 = D$ such that
  $-K_{X} \cdot D_{i-1} < -K_{X} \cdot D_{i}$ and $D_i$ supports multipliers
  for $D_{i-1}$ for all $1 \leqslant i \leqslant j$, and $D_{\!j}$ is nef. If
  $D_{\!j} = 0$ or $D_{\!j}$ is a multiple of a conic bundle, then we are
  done. If $D_{\!j} \neq 0$ and not ample, Lemma~\ref{l:tri} shows that there
  exists a sequence of birational morphisms, contracting a real $(-1)$-curve
  or a conjugate pair of $(-1)$-curves at each step, such that the composition
  $\pi \colon X \to X'$ is a strongly dominant morphism onto a real del Pezzo
  surface of degree greater than $d \coloneq \deg(X)$ and $D_{\!j}$ is the
  pullback under $\pi$ of an ample divisor $D_{\!j}'$ on $X'$. Since $\pi$ is
  strongly dominant, the nonnegative global sections of
  $\mathcal{O}_{\!X}(D_{\!j})$ coincide with the nonnegative global sections
  of $\mathcal{O}_{\!X'}(D'_{\!j})$. Hence, we may work on $X'$ or,
  equivalently, assume that $D_{\!j}$ is ample. If $D_{\!j}$ is ample, the
  Lemma~\ref{l:delAmp} demonstrates that there exists a nonzero effective
  divisor $C$ such that $D_{\!j+1} \coloneq D_{\!j} - C$ is effective and
  supports multipliers for $D_{\!j}$. Since $-K_{X}$ is ample and $C$ is
  effective, we must have $-K_{X} \cdot D_{\!j+1} < -K_{X} \cdot D_{\!j}$.
  This process must terminate after at most $-K_{X} \cdot D$ steps.
\end{proof}

From the algorithm outlined in the proof of Theorem~\ref{t:delPezzo}, we see
that nonnegativity certificates on totally-real del Pezzo surfaces of degree
at least $3$ can be computed via an explicit sequence of semidefinite
programs. The next example shows that these semidefinite programs depend on
the real structure on $X$.

\begin{example} 
  Let $X$ be a real cubic surface in $\PP^3$ and let $f$ be a nonnegative
  global section in
  $\smash{H^0 \kern-1.0pt \bigl( X, \sO_{X}(-2K_{X}) \kern-1.0pt \bigr)}$.
  Equivalently, $f$ is a homogeneous quartic polynomial which is nonnegative
  on $X(\RR)$. To highlight the importance of the real structure, we consider
  on two cases: $\PP^2(6,0)$ and $\mathbb{D}(1,0)$. In both of these cases,
  the surface contains at least one real $(-1)$-curve $C$, so the divisor
  $-K_{X} - C$ supports multipliers for $-K_{X}$. In other words, there exists
  a nonnegative global section $g$ in
  $\smash{H^0 \kern-1.0pt \bigl( X, \sO_{X}(-2K_{X}-2C) \kern-1.0pt \bigr)}$
  and a sum-of-squares $h$ in
  $\smash{H^0 \kern-1.0pt \bigl( X, \sO_{X}(-4K_{X}-2C) \kern-1.0pt \bigr)}$
  such that $f g = h$. Moreover, $B \coloneq -K_{X} - C$ is a real conic
  bundle. When $X = \PP^2(6,0)$, the contraction associated to $B$ is
  surjective on real points, so $g$ is a sum-of-squares of global sections in
  $\smash{H^0 \kern-1.0pt \bigl( X, \sO_{X}(B) \kern-1.0pt \bigr)}$ and
  $f = h/g$. When $X = \mathbb{D}(1,0)$, the contraction associated to $B$
  sends the real points $X(\RR)$ to the disjoint union
  $[a_1,a_2] \sqcup [a_3,a_4]$. As in Remark~\ref{r:modSOS}, the nonnegative
  global section $g$ can be expressed as
  $g_0 + c_1 (x_0 - a_1 \, x_1) (x_0 - a_2 \, x_1) + c_2 (a_0 \, x_1 -
  x_0)(x_0 - a_4 \, x_1)$ where $g_0$ is a sum-of-squares in
  $\smash{H^0 \kern-1.0pt \bigl( X, \sO_{X}(B) \kern-1.0pt \bigr)}$ and
  $c_1,c_2$ are nonnegative real numbers, so we obtain
  \[
    f = \frac{h}{g_0 + c_1 (x_0 - a_1 \, x_1) (x_0 - a_2 \, x_1) + c_2 (a_0 \,
      x_1 - x_0)(x_0 - a_4 \, x_1)} \, .
  \]
  Solely in terms of degree bounds, a nonnegative quartic form on a cubic
  surface admits a quadratic nonnegative multiplier. In the first case (but
  not the second), the multiplier is a sum of squares.
\end{example}

\begin{remark}
  Combining Theorem~\ref{t:delPezzo} and Remark~\ref{r:modSOS}, we see that
  there are three kinds of multiplier certificates on del Pezzo surfaces
  having degree at least $3$. Table~\ref{f:type} summarizes the relationship
  between surface type and nonnegativity certificates.
\end{remark}

\addtocounter{lemma}{1}
\begin{table}[!ht]
  \centering
  \caption{Three kinds of multiplier certificates}
  \label{f:type}
  \vspace{-0.5em}
  \begin{tabular}{ll} \hline
    \multicolumn{1}{c}{Surface type} 
    & \multicolumn{1}{c}{Nonnegativity certificate type} \\ \hline 
    \\[-10pt]
    Admits a birational morphism to $\mathbb{D}$
    & Modified SOS multiplier with $2$ intervals \\
    Admits a birational morphism to $Q^{3,1}(0,2)$
    & Modified SOS multiplier with $1$ interval \\
    Otherwise
    & SOS-multipliers \\[2pt] \hline
  \end{tabular}
\end{table}

Although modified certificates are necessary to characterize nonnegativity for
arbitrary divisors on a del Pezzo surface, we demonstrate that sums of squares
may suffice for a specific divisor.

\begin{example} 
  The surface $Q^{3,1}(0,2)$ embedded via its anticanonical bundle is a
  subvariety $X$ in $\PP^6$ of degree $6$. A nonnegative quadratic form $f$ on
  $X$ is a nonnegative global section of
  $\smash{H^0 \kern-1.0pt \bigl( X, \sO_{X}(-2K_{X}) \kern-1.0pt \bigr)}$. The
  divisor $B \coloneq L_1+L_2-E_1-E_2$ be the unique real conic bundle on $X$.
  Following the algorithm outlined in the proof of Theorem~\ref{t:delPezzo},
  the divisor $D \coloneq -K_X - B = L_1+L_2$ supports multipliers for the
  divisor $-K_X$ and the divisor $0$ supports multipliers for $D$. Hence,
  there exists an equation of the form $fg = h$ where $g$ and $h$ are sums of
  squares. More precisely, $g$ is a sum of squares of linear forms vanishing
  on $B$, $h$ is a sum of squares of quadratic forms vanishing on $B$, and $f$
  admits a quadratic sum-of-squares multiplier.
\end{example}

%%%%%%%%%%%%%%%%%%%%%%%%%%%%%%%%%%%%%%%%%%%%%%%%%%%%%%%%%%%%%%%%%%%%%%%%%%%%%%
\section{Asymptotic Multipliers Bounds}
\label{s:asymp}

\noindent
This section gives asymptotic bounds for the degree of multipliers on certain
embedded surfaces. We provide quadratic upper bounds on the growth rate rather
than exact bounds, because we have more control over the transfer steps than
the base case of the induction. Nevertheless, our novel results constitute the
first multiplier bounds for nonrational surfaces beyond the elementary
recursive degree estimates that apply to all real varieties; compare with
\cite{LPR}*{Theorem~1.5.7}. In hindsight, our methods handle a totally-real
smooth surface $X$ with a very ample divisor $A$ such that
$-K_{X} \cdot A > 0$. Despite apparently aligning with algebraic surfaces of
Kodaira dimension $-\infty$, it is unclear whether this is an artifact of our
techniques or reflects some deeper aspect of nonnegativity certificates.

As in the minimal model program for algebraic surfaces, our approach exploits
the $(-1)$-curves on a surface. We start by showing how a single $(-1)$-curve
can produce $1$-step transfers.

\begin{lemma}
  \label{l:blowup}
  Assume that $X$ is a totally-real smooth surface with a very ample divisor
  $A$. Let $\pi \colon Z \coloneq \Bl_{p}(X) \to X$ be the blow-up of $X$ at a
  real point $p$, let $E \coloneq \pi^{-1}(p)$ be the exceptional divisor, and
  set $H \coloneq \pi^*(A)$. Fix a positive integer $m$ and choose a
  nonnegative integer $\ell$ such that the divisor $\ell H - K_{Z}$ is big and
  nef.
  \begin{compactenum}[\upshape (i)]
  \item For any positive integer $k$ and any integer $d$ satisfying
    $d \geqslant 2m+k + \ell$, the inequality
    \[
      2d (-K_{Z} \cdot H) - (2m+k)(k+1) - \chi(\sO_Z) > 0
    \]
    implies that the divisor $d H - (m+k) E$ supports multipliers on the divisor $d H - m E$.
  \item For any integer $d$ satisfying $d \geqslant 2m + \ell$, the inequality
    \[
      (1-2d) \bigl( H \cdot (H + K_{Z}) \bigr) + m(m-1) - \chi(\sO_Z) > 0
    \]
    implies that the divisor $(d-1) H$ supports multipliers on the divisor $d H - m E$.
  \end{compactenum}
\end{lemma}

Before proving the lemma, we list a few rudimentary properties of the surface
$Z$. Firstly, we have $K_{Z} = \pi^{*}(K_{X}) + E$. Secondly, the divisor
$2H-E$ on $Z$ is very ample; for example see \cite{BS}*{Theorem~2.1}. Thirdly,
for any sufficiently large integer $\ell$, the Nakai Criterion implies that
the divisor $\ell H - K_{Z}$ on $Z$ is big and nef, because
$(\ell H - K_{Z})^2 = \ell^2 A^2 - 2 \ell A \cdot K_{X} + (K_{X})^2 - 1$ and
$A^2 > 0$.

\begin{proof}
  To establish part~(i), it suffices to verify the hypotheses of
  Corollary~\ref{c:simp}. Set $D_{0} \coloneq d H - m E$ and
  $D_{1} \coloneq d H - (m+k) E$. The inequality $d \geqslant 2m + k + \ell$
  means that there exists a nonnegative integer $j$ such that
  $d = 2m+ j + k + \ell$. As the sum of an very ample divisor and a free
  divisor, both $D_{0} = m(2H-E) + (j + k + \ell)H$ and
  $D_{0} + D_{1} = (2m+k)(2H-E) + 2(j+\ell) H$ are very ample and thereby
  free; see \cite{Hartshorne}*{\S II.7, Exercise~7.5d}. Since $k > 0$ and the
  divisor $E$ is effective, the equality $D_{1} - D_{0} = -k E$ shows that
  $h^0( Z, D_{1}-D_{0}) = 0$. The divisor $\ell' H - K_{Z}$ being big and nef,
  for all integers $\ell' \geqslant \ell$, guarantees that, for any positive
  integer $c$, the divisors
  \begin{align*}
    c D_{0} - K_{Z}
    &= cm(2H-E) + \bigl( c(j+k+\ell) H- K_{Z} \big) \, , 
    \\
    c D_{0} + c D_{1} - K_{Z}
    &= (2m+k)(2H-E) + \bigl( 2c(j+\ell) H - K_{Z} \bigr)
  \end{align*}
  are also big and nef. Hence, the Kawamata--Viehweg Vanishing Theorem gives
  $h^i(Z, cD_{0}) = 0$ and $h^i( Z, cD_{0} + cD_{1} \bigr) = 0$ for any
  positive integers $i$ and $c$. All that remains is confirm the inequality
  $\chi(2D_{1}) + h^1(Z, D_{1} - D_{0}) > \chi(-D_{0} - D_{1})$. Applying the
  Riemann--Roch Formula, we deduce that
  \begin{align*}
    &\relphantom{=} \chi(2D_{1}) + h^1(Z, D_{1} - D_{0}) - \chi(-D_{0} -
      D_{1}) \\[-2pt]
    &= \chi(2D_{1}) - \chi(D_{1}-D_{0}) + h^2(X, D_{1}-D_{0}) -
      \chi(-D_{0}-D_{1}) \\[-2pt]
    &\geqslant \tfrac{1}{2} \bigl( (2D_{1})^2 - 2D_{1} \cdot K_{Z} -
      (D_{1}-D_{0})^2 + (D_{1}-D_{0}) \cdot K_{Z} - (D_{0}+D_{1})^2 -
      (D_{0}+D_{1}) \cdot K_{Z} \bigr) - \chi(\sO_{Z}) \\[-2pt]
    &= D_1^2 - D_1 \cdot K_{Z} - D_{0}^2 - D_0 \cdot K_{Z} - \chi(\sO_{Z})
    \\[-2pt]
    &= d^2H^2 - (m+k)^2 -d H \cdot K_{Z} - m - d^2H^2 + m^2 - d H \cdot K_{Z}
      - (m+k) - \chi(\sO_{Z}) \\[-2pt]    
    &= 2d(-K_{Z} \cdot H) - (2m+k)(k+1) - \chi(\sO_{Z}), 
  \end{align*}
  which is positive by assumption.

  For part~(ii), it again suffices to verify the hypotheses of
  Corollary~\ref{c:simp}. Set $D_0 \coloneq d H - m E$ and
  $D_{1} \coloneq (d-1)H$. The inequality $d \geqslant 2m + \ell$ means that
  there exists a nonnegative integer $j$ such that $d = 2m + j + \ell$. As the
  sum of an very ample and a free divisor, both $D_{0} = m(2H-E) + (j+\ell)H$
  and $D_{0} + D_{1}$ are very ample and free. Since
  $H \cdot (D_1 - D_0) = H \cdot (m E - H) = - H^2 < 0$, the divisor
  $D_1 - D_0$ is not effective and $h^0(Z, D_{1}-D_{0}) = 0$. The divisor
  $\ell' H - K_{Z}$ being big and nef, for all integers
  $\ell' \geqslant \ell$, also guarantees that, for any positive integer $c$,
  the divisors
  \begin{align*}
    c D_{0} - K_{Z}
    &= cm(2H-E) + \bigl( c(j+\ell)H - K_{Z} \bigr) \, , 
    \\
    c D_{0} + c D_{1} - K_{Z}
    &= cm(2H-E) + \bigl( c(2m+2j+\ell-2) H - K_{Z} \bigr) 
  \end{align*}
  are big and nef. Hence, the Kawamata--Viehweg Vanishing Theorem gives
  $h^i(Z, cD_{0}) = 0$ and $h^i( Z, cD_{0} + cD_{1} \bigr) = 0$ for any
  positive integers $i$ and $c$. As in part~(i), it remains to confirm that
  $\chi(2D_{1}) + h^1(Z, D_{1} - D_{0}) > \chi(-D_{0} - D_{1})$. Applying the
  Riemann--Roch Formula, we deduce that
  \begin{align*}
    \chi(2D_{1}) + h^1(Z, D_{1} - D_{0})
    &- \chi(-D_{0} - D_{1}) \geqslant  D_1^2 - D_1 \cdot K_{Z} - D_{0}^2 - D_0
      \cdot K_{Z} - \chi(\sO_{Z}) \\[-2pt]
    &= (d-1)^2 H^2 - (d-1) H \cdot K_{Z} - d^2 H^2 + m^2 - d H \cdot K_{Z} - m
      - \chi(\sO_{Z}) \\[-2pt]
    &= (1-2d)\bigl( H \cdot (H+K_{Z}) \bigr) + m(m-1) - \chi(\sO_{Z}) 
  \end{align*}
  which is again positive by assumption. 
\end{proof}

\begin{remark}
  \label{r:app}
  The hypotheses in Lemma~\ref{l:blowup} constrain the underlying surface $X$.
  To have the inequality $2d (-K_{Z} \cdot H) - (2m+k)(k+1) - \chi(\sO_Z) > 0$
  hold for some choice of positive integers $m$ and $k$ and any sufficiently
  large integer $d$ requires $-K_{Z} \cdot H > 0$. As
  $-K_{Z} \cdot H = - K_{X} \cdot A$ and $A$ is a very ample divisor on $X$,
  it follows that no multiple of $K_{X}$ can be effective, so
  $h^0(X, c K_{X}) = 0$ for any positive integer $c$. Hence, the Enriques
  characterization~\cite{Bea}*{Theorem~VI.17} implies that $X$ must be a
  \emph{ruled surface}: birationally equivalently to a product
  $C \times \PP^1$ for some nonsingular curve $C$. Furthermore, the condition
  $-K_{X} \cdot A > 0$ is the same as $A^2 > A \cdot (A +K_{X})$. Assuming
  that the surface $X$ is embedded into projective space via the complete
  linear series associated to $A$, the adjunction
  formula~\cite{Bea}*{Remarks~I.16} implies that the genus
  $\operatorname{g}(X,A)$ of a general hyperplane section is
  $\tfrac{1}{2} \bigl( A \cdot (A + K_{X}) \bigr) + 1$. Hence, the degree of
  this embedded surface is greater than $2 \operatorname{g}(X,A) - 2$.
\end{remark}

The next theorem comes from repeated use of the $1$-step transfers arising
from $(-1)$-curves.

\begin{theorem} 
  \label{t:ruled} 
  Assume that $X$ is a totally-real smooth surface with a very ample divisor
  $A$ satisfying $-K_{X} \cdot A > 0$. Let
  $\pi \colon Z \coloneq \Bl_{p}(X) \to X$ be the blow-up of $X$ at a real
  point $p$ and set $H \coloneq \pi^*(A)$. Fix $s$ to be the smallest positive
  integer such that $s (-K_{X} \cdot A) > A \cdot (A + K_{X})$ and choose a
  positive integer $t$ such that
  $\tfrac{1}{2} + \tfrac{1}{3} + \dotsb + \tfrac{1}{t+1} > 2(1 + \sqrt{s})$.
  For all sufficiently large integers $d$, there exists an $(t+1)$-step
  transfer on $Z$ from $d H$ to $(d-1) H$.
\end{theorem}

The harmonic series being divergent affirms the existence of the positive
integer $t$.

\begin{proof} 
  Let $E \coloneq \pi^{-1}(p)$ be the exceptional divisor on $Z$. We claim
  that, for any sufficiently large integer $d$, there exist positive integers
  $m_0, m_1, \dotsc, m_t$ such that the divisors
  \begin{align*}
    D_0 &\coloneq d H \, , 
    & D_1 &\coloneq  d H - m_1 E \, ,
    & D_2 &\coloneq  d H - m_2 E \, , 
    &&\dotsc \,  ,
    & D_{t} &\coloneq  d H - m_t E \, ,
    & D_{t+1} &\coloneq  (d-1) H
  \end{align*}
  form an $(t+1)$-step transfer from $dH$ to $(d-1)H$. Consider the function
  $\Lambda \colon \NN \to \ZZ$ defined by
  $\Lambda(d) = 2d(-K_{Z} \cdot H) - \chi(\sO_{Z})$. Since $t$ depends only on
  $s$ (and not $d$), this function enjoys the following three properties.
  First, for any integer $j$ satisfying $1 \leqslant j \leqslant t$, there
  exists a positive integer $k_{\!j}$ such that, for any sufficiently large
  integer $d$, we have
  \[
    \frac{1}{2(j+1)} \sqrt{\Lambda(d)}
    \leqslant k_{\!j} \leqslant \frac{1}{2j} \sqrt{\Lambda(d)} - 1 \, . 
  \]
  Second, for any positive integer $\ell$ such that the divisor
  $\ell H - K_{Z}$ on $Z$ is big and nef, and any sufficiently large integer
  $d$, we have
  \[
    \left( 1 + \frac{1}{2} + \frac{1}{3} + \dotsb + \frac{1}{t+1} \right)
    \sqrt{\Lambda(d)} < d - \ell \, .
  \]
  Third, for any sufficiently large integer $d$, we have
  $d > - H \cdot (H + K_{Z})$ and $d > - (s+1) \chi(\sO_{Z})$.
  Assume that the integer $d$ is large enough that both of these properties
  hold. Set $m_{0} \coloneq 0$ and, for all $1 \leqslant j \leqslant t$, set
  $m_{j} \coloneq \sum_{i=1}^{j} k_{\!j}$. For any
  $1 \leqslant j \leqslant t$, the two properties give
  \begin{align*}
    2 m_{j} 
    &\leqslant 2 \sum_{i=1}^{j} \left( \frac{1}{2j} \sqrt{\Lambda(d)} -1
      \right)
      \leqslant \left( 1 + \frac{1}{2} + \frac{1}{3} + \dotsb + \frac{1}{t+1}
      \right) \sqrt{\Lambda(d)}
      < d - \ell \, ,
  \end{align*}
  so we obtain $d > 2m_{j} + \ell > 2m_{j-1} + k_{j} + \ell$. Since
  $k_{j} + 1 \leqslant \frac{1}{2j} \sqrt{\Lambda(d)}$, we also have
  \begin{align*}
    (2 m_{j-1} +k_{j})(k_{j} + 1)
    &\leqslant 2 m_{j-1} (k_{j} + 1) + (k_{j}+1)^2 \\[-6pt]
    &\leqslant \left[ \left( 1 + \frac{1}{2} + \frac{1}{3} + \dotsb +
      \frac{1}{j-1}  \right) \sqrt{\Lambda(d)}
      \vphantom{\frac{\sqrt{\Lambda(d)}}{2j}} \right] \! \left[
      \frac{\sqrt{\Lambda(d)}}{2j} \right] + \left[
      \frac{\sqrt{\Lambda(d)}}{2j} \right]^{\!2} \\[-2pt]
    &< \left( \frac{j-1}{2j} + \frac{1}{4j^2} \right) \Lambda(d) < \Lambda(d)
      = 2d(-K_{Z} \cdot H) - \chi(\sO_{Z}) \, .
  \end{align*}
  Thus, Lemma~\ref{l:blowup}.i demonstrates that, for all
  $1 \leqslant j \leqslant t$, the divisor $D_{\!j}$ supports multipliers for
  the divisor $D_{\!j-1}$. For the last transfer, the choice of $t$ and the
  first property give
  \begin{align*}
    m_{t} - 1 
    &\geqslant \frac{1}{2} \left( \frac{1}{2} + \frac{1}{3} + \dotsb +
      \frac{1}{t+1} \right) \sqrt{\Lambda(d)} -1
      \geqslant (1+\sqrt{s}) \sqrt{\Lambda(d)} - 1 > \sqrt{s \, \Lambda(d)}
      \, ,
  \end{align*}
  so we obtain
  $m_{t}(m_{t}-1) > (m_{t}-1)^2 > s \, \Lambda(d) = 2ds(-K_{Z}\cdot H) - s
  \chi(\sO_{Z})$. Combining the inequality
  $s (-K_{Z} \cdot H) = s (-K_{X} \cdot A) > A \cdot (A + K_{X}) = H \cdot (H
  + K_{Z})$ with the third property, it follows that
  \begin{align*}
    (1-2d) \bigl( H \cdot (H + K_{Z}) \bigr)
    &+ m_{t}(m_{t}-1) - \chi(\sO_{Z}) \\[-2pt]
    &> (1-2d) \bigl( H \cdot (H + K_{Z}) \bigr) + 2ds(-K_{Z}\cdot H) - (s+1)
      \chi(\sO_{Z}) \\[-2pt]
    &= H \cdot (H + K_{Z}) + 2d \bigl( s (-K_{Z} \cdot H)
      - H \cdot (H + K_{Z})  \bigr) - (s+1) \chi(\sO_{Z}) > 0 \, .
  \end{align*}
  Therefore, Lemma~\ref{l:blowup}.ii proves that the divisor $D_{t+1}$
  supports multipliers for the divisor $D_{t}$.
\end{proof}

\begin{remark} 
  Both Example~\ref{e:hb2} and the proof of Theorem~\ref{t:ruled} use the same
  inherent strategy. A more detailed understanding of $\Bl_{p}(\PP^2)$ in the
  first case is the only substantial difference.
\end{remark}

From this theorem, we extract a quadratic upper bound on the growth rate of
the degree of multipliers on embedded ruled surfaces.

\begin{corollary} 
  \label{c:quad} 
  Let $X$ be a totally-real smooth surface with a very ample divisor $A$
  satisfying $-K_{X} \cdot A > 0$. For any positive integer $d$ and any
  nonnegative global section $f$ in
  $\smash{H^0 \kern-1.0pt \bigl( X, \sO_{X}(2dA) \kern-1.0pt \bigr)}$, there
  exists a nonzero sum-of-squares global section $g$ in
  $\smash{H^0 \kern-1.0pt \bigl( X, \sO_{X}(2eA) \kern-1.0pt \bigr)}$ such
  that $e = O(d^2)$ and the product $f g$ is a sum of squares.
\end{corollary}

\begin{proof} 
  Let $\pi \colon Z \coloneq \Bl_{p}(X) \to X$ be the blow-up of $X$ at a real
  point $p$ and set $H \coloneq \pi^*(A)$. By repeated applications of
  Theorem~\ref{t:ruled}, there exists an integer $d_0$ such that, for any
  $d > d_0$, there exists a $(t+1)(d-d_0)$-step transfer
  $D_0, D_1, \dotsc, D_{(t+1)(d-d_0)}$ on $Z \coloneq \Bl_{p}(X)$ from the
  divisor $d H$ to the divisor $d_0 H$. Set $r \coloneq (t+2)(d-d_0)$. In
  particular, for any nonnegative global section $f_0$ in
  $\smash{H^0 \kern-1.0pt \bigl( X, \sO_{X}(2dA) \kern-1.0pt \bigr) = H^0
    \kern-1.0pt \bigl( Z, \sO_{X}(2dH) \kern-1.0pt \bigr)}$, there exists
  nonnegative global sections $f_{i}$ in
  $\smash{H^0 \kern-1.0pt \bigl( Z, \sO_{Z}(2D_i) \kern-1.0pt \bigr)}$ and
  sums of squares $g_{i-1,i}$ in
  $\smash{H^0 \kern-1.0pt \bigl( Z, \sO_{Z}(2D_{i-1} + 2D_{i}) \kern-1.0pt
    \bigr)}$ such that $f_{i-1} f_{i} = g_{i-1,i}$ for all
  $1 \leqslant i \leqslant r$. Furthermore, from the proof of
  Theorem~\ref{t:ruled}, we see that the first $(t+1)$ global sections $f_i$
  correspond to forms of degree $2d$ on $X$ and the next $(t+1)$ global
  sections $f_i$ correspond to forms of degree $2(d-1)$ on $X$. Continuing
  this pattern, the degrees of the corresponding forms weakly decrease until
  the very last form $f_{r}$ has degree $2d_{0}$. Invoking
  \cite{LPR}*{Theorem~1.3.2}, the nonnegativity of $f_{r}$ on $X$ implies that
  there exists sums of squares $g'$ and $g''$ such that
  $g' f_{(t+1)(d-d_0)} = g''$ where the degrees of $g'$ and $g''$ are bounded
  above by a constant $e_0$ which depends only on $d_0$. If $r$ is odd, then
  we have
  $f_0 ( g_{1,2} g_{3,4} \dotsb g_{r-2,r-1} g'') = f_0 f_1 \dotsb f_r g' =
  g_{0,1} g_{2,3} \dotsb g_{r-1,r} g''$. If $r$ is odd, then we have
  $f_0(g_{1,2} g_{3,4} \dotsb g_{r-1,r} g') = f_0 f_1 \dotsb f_r g' = g_{0,1}
  g_{2,3} \dotsb g_{r-2,r-1} g''$. Therefore, we deduce that $f_0$ has of
  sum-of-square multiplier of degree $O(d^2)$.
\end{proof}

Making stronger assumptions on the underlying surface allows for a more
streamlined conclusion.

\begin{proposition} 
  \label{p:elliptic}
  Assume that $X$ is a nondegenerate nonrational totally-real smooth surface
  with a very ample divisor $A$ such that its sectional genus
  $\operatorname{g}(X,A)$ equals $1$. Let
  $\pi \colon Z \coloneq \Bl_{p}(X) \to X$ be the blow-up of $X$ at a real
  point $p$ and set $H \coloneq \pi^*(A)$. For any integer $d$ satisfying
  $d \geqslant 5$, there exists a $2$-step transfer on $Z$ from $dH$ to
  $(d-1)H$.
\end{proposition}

\begin{proof} 
  The genus formula is $2 \operatorname{g}(X,A) - 2 = A \cdot (A + K_{X})$, so
  the assumption that $\operatorname{g}(X,A) = 1$ is equivalent to
  $0 = A \cdot (A + K_{X}) = H \cdot (H + K_{Z})$. Combined with the
  nondegeneracy of $X$, we deduce that $- K_{Z} \cdot H \geqslant 2$. As in
  Remark~\ref{r:app}, it follows that $Z$ is a ruled surface birational to
  $C \mathbin{\times} \PP^1$ for some nonsingular curve $C$. The Euler
  characteristic of a surface being a birational
  invariant~\cite{Bea}*{Proposition~III.20} implies that
  $\chi(\sO_{Z}) = \chi(\sO_{C \mathbin{\times} \PP^1}) = 1 -
  \operatorname{g}(C)$. Hence, the hypothesis that the surface $X$ is not
  rational gives $- \chi(\sO_{Z}) \geqslant 0$.

  Let $E \coloneq \pi^{-1}(p)$ be the exceptional divisor on $Z$. When
  $d \geqslant 5$, we claim that the divisors $D_0 \coloneq d H$,
  $D_{1} \coloneq dH - 2E$, and $D_{2} \coloneq (d-1)H$ form a $2$-step
  transfer on $Z$. Since the divisor $H - K_{Z}$ is big and nef on $Z$ and
  $d \geqslant 5$, we see that
  $2d (-K_{Z} \cdot H) - 2(2+1) - \chi(\sO_{Z}) \geqslant 4(5)-6 > 0$. Thus,
  Lemma~\ref{l:blowup}.i establishes that the divisor $D_{1} = dH-2E$ supports
  multipliers on the divisor $D_{0} = dH$. Similarly, as $d \geqslant 2$ and
  $(1-2d) \bigl( H \cdot (H + K_{Z}) \bigr) + 2(2-1) - \chi(\sO_Z) \geqslant 2
  > 0$, Lemma~\ref{l:blowup}.ii establishes that the divisor $D_{2} = (d-1)H$
  supports multipliers on the divisor $D_{1} = dH-2E$.
\end{proof}

We finish with a couple of examples. Consider the ruled surface
$X = C \times \PP^{1}$ and let $\pi_1 \colon X \to C$ be the canonical
projection onto the smooth curve $C$. Choose a fibre $F$ of $\pi_{1}$ and a
section $C$ (by a slight abuse of notation). The Picard group of $X$ is
generated by the class of $C$ and pullbacks of elements from the Picard group
of the curve $C$, and the N{\'e}on--Severi group of $X$ is generated by $C$
and $F$; see \cite{Hartshorne}*{Proposition~V.2.3} or
\cite{Bea}*{Proposition~III.18}. Moreover, we have $C \cdot F = 1$,
$C^2 = F^2 = 0$, and $K_{X} \equiv -2C + \omega F$ where $\omega$ is the
canonical divisor on $C$; see \cite{Hartshorne}*{Lemma~V.2.10}.

\begin{example}
  \label{e:ellRul} 
  For a totally-real elliptic curve $C$, let $X$ be the ruled surface
  $C \mathbin{\times} \PP^1$. Choose a point $p$ on $C$ and consider a divisor
  $A \equiv C + (3p) F$. This divisor is very ample and coincides with the
  Segre embedding of $X$ into $\PP^{5}$ as a surface of degree $6$. The
  canonical divisor is $K_{X} = -2C$, so $A \cdot (A + K_{X}) = 0$. It follows
  that the sectional genus $\operatorname{g}(X,A)$ is $1$. Therefore,
  Proposition~\ref{p:elliptic} establishes that, for all $d \geqslant 5$,
  there is a $2$-step transfer from $d A$ to $(d-1) A$.
\end{example}

\begin{example} 
  Let $C$ be a totally-real nonsingular curve of genus $g$ that is not
  hyperelliptic and let $X$ be the ruled surface
  $X = C \mathbin{\times} \PP^1$. For an integer $m > 0$, consider the divisor
  $A \coloneq m (C + \omega F)$. The divisor $A$ is very ample because it is a
  multiple of $C + \omega F$ which is the pullback of hyperplane class under
  the Segre embedding of the closed immersion
  $\kappa \times \operatorname{id} \colon C \times \PP^1 \to \PP^{g-1}\times
  \PP^1$ where $\kappa$ is the canonical embedding of $C$. Since
  $-K_{X} \cdot A = m \deg(\omega) > 0$, Theorem~\ref{t:ruled} produces
  asymptotic transfer results for such surfaces. From the equality
  $A \cdot (A+K_{X}) = (2m^2-m) \deg(\omega)$, we see that the ration
  $(A \cdot (A + K_{X})/ (-K_{X} \cdot A) = 2m-1$ can be made arbitrarily
  large. We conclude that the number of transfer steps required to pass from
  $d A$ to $(d-1)A$ via Theorem~\ref{t:ruled} is not uniformly bounded on all
  surfaces satisfying $-K_{X} \cdot A > 0$.
\end{example}

\subsection*{Acknowledgements}

Grigoriy Blekherman was partially supported by NSF grant DMS-1901950. Gregory
G.~Smith was partially supported by NSERC, the NSF grant DMS-1928930 while in
residence at the Simons Laufer Mathematical Sciences Institute, and the Knut
and Alice Wallenberg Foundation while visiting the Royal Institute of
Technology (KTH). Mauricio Velasco was partially supported by ANII grant
FCE-1-2023-1-176172.

%%%%%%%%%%%%%%%%%%%%%%%%%%%%%%%%%%%%%%%%%%%%%%%%%%%%%%%%%%%%%%%%%%%%%%%%%%%%%%

\raggedright

\end{document}